\documentclass[a4paper, 11pt]{amsart}
\usepackage{amssymb,array}
\usepackage{url}
\usepackage{graphicx}
\usepackage{subfigure}
\usepackage{enumerate}

\usepackage[left=3cm, right=3cm, top=3cm, bottom=3cm]{geometry}

\usepackage{bbm}

\renewcommand{\leq}{\leqslant}
\renewcommand{\geq}{\geqslant}

\newcommand{\N}{\mathbb{N}}

\newcommand{\R}{\mathbb{R}}
\newcommand{\C}{\mathbb{C}}

\renewcommand{\P}{\mathbb{P}}

\newcommand{\U}{\mathcal{U}}

\newcommand{\M}{\mathcal{M}}

\newcommand{\module}[1]{\left| #1\right|}

\newcommand{\norm}[1]{\left\Vert #1\right\Vert}
\newcommand{\normt}[1]{\left\Vert #1\right\Vert_{(t)}}

\DeclareMathOperator{\trace}{Tr}

\DeclareMathOperator{\I}{I}
\DeclareMathOperator{\id}{id}

\DeclareMathOperator{\Gr}{Gr}

\DeclareMathOperator{\Image}{Im}

\DeclareMathOperator{\diag}{diag}
\DeclareMathOperator{\spn}{span}

\newcommand{\isom}{\simeq}
\newcommand{\scalar}[2]{\langle #1 , #2\rangle}
\newcommand{\e}{\varepsilon}
\renewcommand{\phi}{\varphi}
\newcommand{\iy}{\infty}
\renewcommand{\d}[1]{\mathrm{d}#1}
\newcommand{\ind}{\mathbbm{1}}

\newcommand{\Rnc}{\mathbb{R}_{\neq}}

\newtheorem{theorem}{Theorem}[section]
\newtheorem{definition}[theorem]{Definition}
\newtheorem{proposition}[theorem]{Proposition}
\newtheorem{remark}[theorem]{Remark}
\newtheorem{lemma}[theorem]{Lemma}

\newtheorem{corollary}[theorem]{Corollary}

\begin{document}

\title[Eigenvectors and eigenvalues in a random subspace of a tensor product]{Eigenvectors and eigenvalues in a random subspace of a tensor product}
\author{Serban Belinschi}
\address{Institute of Mathematics ``Simion Stoilow'' of the Romanian Academy and Department of Mathematics \& Statistics,
University of Saskatchewan;
106 Wiggins Road
Saskatoon, SK S7N 5E6 } 
\email{ belinschi@math.usask.ca}
\author{Beno\^{\i}t Collins}
\address{
D\'epartement de Math\'ematique et Statistique, Universit\'e d'Ottawa,
585 King Edward, Ottawa, ON, K1N6N5 Canada
and 
CNRS, Institut Camille Jordan Universit\'e  Lyon 1, 43 Bd du 11 Novembre 1918, 69622 Villeurbanne, 
France} 
\email{bcollins@uottawa.ca}
\author{Ion Nechita}
\address{
D\'epartement de Math\'ematique et Statistique, Universit\'e d'Ottawa,
585 King Edward, Ottawa, ON, K1N6N5 Canada} 
\email{inechita@uottawa.ca}
\subjclass[2000]{Primary 15A52; Secondary 52A22, 46L54 }
\keywords{Random matrices, Random projections, Singular values of random vectors, free additive convolution}

\begin{abstract}

Given two positive integers $n$ and $k$ and a parameter $t\in (0,1)$, we choose at random a vector subspace $V_{n}\subset \mathbb{C}^{k}\otimes\mathbb{C}^{n}$
of dimension $N\sim tnk$. We show that the set of $k$-tuples of singular values of all unit vectors in $V_n$ fills asymptotically (as $n$ tends to infinity) a deterministic convex set $K_{k,t}$ that we describe using a new norm in $\R^k$.

Our proof relies on free probability, random matrix theory, complex analysis and matrix analysis techniques. The main result result comes together with a law of large numbers for the singular value decomposition of the eigenvectors corresponding to large eigenvalues of a random truncation of a matrix with high eigenvalue degeneracy.

\end{abstract}

\maketitle

\section{Introduction}

In \cite{collins-nechita-2}, it was observed that if one takes at random a vector subspace $V_{n}$ of $\C^{k}\otimes \C^{n}$ of
relative dimension $t$ for large $n$ and fixed $k$, with very high probability, some sequences of numbers in $\R_{+}^{k}$
never occur as singular values of elements in $V_{n}$ as $n$ becomes large. This result was used to provide a systematic understanding
of some non-additivity theorems for entropies in Quantum Information Theory. We refer to the bibliography of 
\cite{collins-nechita-2} for more information on this topic.

Our aim in this paper is to provide a definitive answer to the question of which sequences of numbers in $\R_{+}^{k}$ occur or not
as singular values of elements in $V_{n}$. 
Our main result can be sketched as follows - for the statement with complete definitions, we refer to Theorem \ref{thm:main}: 
\begin{theorem}\label{thm:main-short}
Let $t\in (0,1)$ be a parameter and for any $n$, $V_{n}$ a vector subspace  of $\C^k\otimes \C^n$ of dimension $N\sim tnk$ chosen at random.
Then, there exists a compact set $K_{k,t} \subset \R^k_+$ such that any $k$-tuple $\lambda$ in the interior of $K_{k,t}$ occurs with high probability as the singular value vector of some norm one vector $x \in V_n$. Moreover, the probability that some vector $\nu \notin K_{k,t}$ occurs as the singular value vector of some element $y \in V_n$ is vanishing when $n \to \infty$.
\end{theorem}

The statement of the above theorem, as well as any other result in this paper about singular values of vectors in a tensor product space, can be immediately translated into a statement about singular values of matrices, simply by fixing an isomorphism $\C^k \otimes \C^n \isom \M_{k \times n}(\C)$; note that the euclidean norm on $\C^k \otimes \C^n$ is pushed into the Schatten 2-norm on $\M_{k \times n}(\C)$, i.e. $\|X\| = \sqrt{\trace(XX^*)}$.
\begin{theorem}
Let $t\in (0,1)$ be a parameter and for any $n$, $V_{n}$ a vector subspace  of $\M_{k \times n}(\C)$ of dimension $N\sim tnk$ chosen at random.
Then, there exists a compact set $K_{k,t} \subset \R^k_+$ such that any $k$-tuple $\lambda$ in the interior of $K_{k,t}$ occurs with high probability as the singular value vector of a matrix $x \in V_n$ of Hilbert-Schmidt norm one. Moreover, the probability that some vector $\nu \notin K_{k,t}$ occurs as the singular value vector of some Hilbert-Schmidt norm one
matrix $y \in V_n$ is vanishing when $n \to \infty$.
\end{theorem}
Even though both formulations are completely equivalent, they are of interest to different areas of mathematics. We choose to work with singular values (or \emph{Schmidt coefficients} as they are called in quantum information) of vectors because of the initial quantum information theoretical motivation.

The set $K_{k,t}$ is described with the help of a new norm on $\R^{k}$, that arises from free probability theory. 
Restricted on $\R_{+}^{k}$, it interpolates between the $l^{1}$ and the $l^{\infty}$ norm.

For the purpose of proving the above theorem, 
one first key technical result (Theorem \ref{thm:RMT-norm}) is a partial extension of a result of
Haagerup and Thorbj\o rnsen \cite{HT} to the case of random projections.
The characterization of sequences  that fail with high probability to occur as singular values of elements in $V_n$
follows from this technical result. It  uses ideas that have been introduced in \cite{collins-nechita-2}.

The characterization of sequences that occur with high probability as singular values of elements in $V_n$ 
is much more involved (we refer to this part of the proof of the main theorem as the proof of the second inclusion, whereas 
we refer to the previous part as the first inclusion). 
It turns out to rely not only on our first technical result, but also on a precise understanding of the eigenvectors 
of suitable random matrix models.

In Random Matrix Theory, the asymptotic behavior of large random matrices is the main object of study, and
 the empirical distributions of the eigenvalues as a random set is arguably the most studied 
 kind of statistics, together with, more recently, the statistics of the largest eigenvalues. 
 To our knowledge, the eigenvectors had not been recognized so far as 
variables having a structured asymptotic behavior (with a few exceptions in the case of spiked random matrices, see e.g. \cite{benaych-rao} and references therein), although they have recently been studied for various models of random matrices (see \cite{benaych} for a recent work in this direction). 

For the purposes of the proof of the second inclusion, we present in this paper
a theorem that is of independent interest, as it shows that the eigenvectors of some random matrices are much more deterministic
than one might expect. 
Our theorem can be summarized as follows ($\U(k)$ denotes the group of $k \times k$ unitary matrices):
\begin{theorem}
Let $A$ be a $k\times k$ positive semidefinite matrix with simple eigenvalues. 
Let $\nu_{n}$ be a sequence of numbers satisfying $\nu_{n}=o(n)$, and $N\sim tnk$ (where $t\in (0,1)$). 
Let $Z_{n}=\Pi_{n} (A\otimes \I_{n}) \Pi_{n}$ where $\Pi_{n}$ is a random projection of rank $N$.
Let $y_{n}$ be the eigenvector corresponding to the $\nu_{n}$-th largest eigenvalue of $Z_n$. Then, almost surely as $n\to\infty$, the
$(\R^{k},\U(k)/\U(1)^{k})$ part of the singular value decomposition of $y_n$ converges to a limit made explicit in 
Theorem \ref{thm:lln-largest-eigenvector}.
\end{theorem}

Finally, we study the points at the boundary of the set $K_{k,t}$ in Theorem \ref{thm:main-short}. The boundary of the dual set is a real algebraic variety for small enough values of $t$, when intersected with the hyperplane $\sum \lambda_{i}=1$. In particular, we show that for some parameters $t$ it is strictly convex, and study its faces for other values of $t$. Our techniques here rely on free probability theory, complex and convex analysis. 

The paper is organized as follows.
In section \ref{sec-setup}, we introduce our model as well as some notation. 
Then, in section \ref{sec:t-norm} we introduce a new norm via an operator algebraic construction and prove a continuity result that we
use in  section \ref{sec:rmt} to prove a convergence result for the norm of the product of random matrices. 
Section \ref{sec:main} is the main section of our paper, where we describe the limiting shape of the collection of singular values. 
In Section \ref{sec:dual-ball}, we study the set $K_{k,t}$ and its dual.

\section{Setup and notations}
\label{sec-setup}

\subsection{Singular values of a vector subspace of a tensor product}

The purpose of this paragraph is to introduce a
subset  $K_{V}\subset \R^{k}$ associated to a vector subspace $V$ of a tensor product $\C^k\otimes \C^n$. We always assume that $k$ and $n$ are integers, with $k \leq n$. This set is a `local' invariant of the inclusion 
$V\subset \C^k\otimes \C^n$ in the sense that it is not modified if $V$ is modified by a unitary in $\U(k)\otimes \U(n)$.

The \emph{singular values} of a vector $x \in \C^k\otimes \C^n$ are non-negative numbers $\lambda_{1}(x)\geq\ldots \geq \lambda_{k}(x)\geq 0$ such that
\begin{equation}\label{eq:vector-SVD}
	x=\sum_{i=1}^{k}\sqrt{\lambda_{i}(x)} \, e_{i}(x)\otimes f_{i}(x)
\end{equation}
where $e_{i}(x)$ (resp.~ $f_{i}(x)$) are orthonormal vectors in $\C^k$ (resp.~ $\C^n$). These are the singular values of the matrix obtained by identifying a vector $x \in \C^k \otimes \C^n$ with the $k \times n$ matrix obtained from $x$ via the isomorphism $\C^k \otimes \C^n \isom (\C^k)^* \otimes \C^n = \M_{k \times n}(\C)$. If $x$ is a unit norm vector in $\C^{nk}$, then $\lambda (x)=(\lambda_{1}(x), \ldots,  \lambda_{k}(x))$ belongs to the set
\begin{equation}
	\Delta_k^\downarrow=\{y\in \R^{k}_{+} \colon  y_1\geq y_2\geq \cdots\geq y_k \text{ and } \sum_{i=1}^k y_i = 1\}.
\end{equation}
We have $\Delta_k^\downarrow\subset\Delta_k$, where $\Delta_k = \{y \in \R_+^k \colon \sum_{i=1}^k y_i = 1\}$ is the $(k-1)$-dimensional probability simplex. We define the following particular vectors
\begin{equation}
	1^j0^{k-j} = (\underbrace{1, 1, \ldots, 1}_{j \text{ times}},\underbrace{0, 0, \ldots, 0}_{k-j \text{ times}}) \in \R^k.
\end{equation}

We also introduce the set $\Rnc^k = \R^k \setminus \R1^k = \{x \in \R^k \colon \exists \, i,j \text{ with } x_i \neq x_j\}$ of vectors with non constant coordinates. Let $V$ be a subspace of dimension $N$ of $\C^k\otimes \C^n$, i.e. an
element of the Grassmann manifold  $\Gr_N(\C^k\otimes \C^n)$.
Let $K_{V}$ be the set of all singular values of norm one vectors $x\in V$,
\begin{equation}
	K_V = \{\lambda(x) \colon x \in V, \|x\|=1\} \subset \Delta_k^\downarrow.
\end{equation}

For technical reasons it will sometimes be convenient to replace it by $\tilde K_{V}$ which is its \emph{symmetrized version} under permuting the coordinates, $\tilde K_{V}$ being a subset of $\Delta_k$:
$$\tilde K_{V} = \{(\lambda_{\sigma(1)}, \lambda_{\sigma(2)}, \ldots, \lambda_{\sigma(k)}) \colon \lambda \in K_V, \sigma \in S_k\}.$$

An elementary but important property of $K_{V}$ is that it has nice invariance properties. The following result is an easy consequence of the singular value decomposition.
\begin{proposition}
$K_{V}$ is invariant under `local' rotations, i.e. if $U_{1}\in \U(k), U_{2}\in \U(n)$ then 
$$K_{V}=K_{(U_{1}\otimes U_{2})\cdot V}.$$
\end{proposition}

\subsection{Random Subspaces}

The integer $k$ and the real parameter $t\in (0,1)$ are fixed throughout the paper.
We are interested in a random sequence $(V_{n})_{n\geq 1}$ of subspaces of $V_{n}\subset\C^k \otimes \C^n$ having the following properties:

\begin{enumerate}
\item $V_{n}$ has dimension $N$ less than $nk$. 
$N$ is a function of $n$ such that $N$ and $n$ grow to infinity according to $N\sim tnk$.
\item The law of $V_{n}$ follows the only probability distribution on the Grassmann manifold $\Gr_N(\C^{k}\otimes\C^n)$ that
is invariant under the action of the unitary group 
$\mathcal U(nk)$. We will refer to this probability measure as the \emph{invariant measure}.
\end{enumerate}

We do not make any assumption about the correlation between the $V_{n}$'s for various values of $n$.
Whether they are correlated or independent does not affect our results. 

In this setting, we call
$$K_{n,k,t}=\tilde K_{V_{n}}$$
and we study the sequence
$K_{n,k,t}$ of symmetrical random subsets of $\Delta_{k}$, as $n\to\infty$.
The aim of this paper is to prove that $K_{n,k,t}$ exhibits a deterministic behavior as $n\to\infty$.
In order to describe it, we need to review a few notions of free probability theory and complex analysis.

\section{Freeness and a new family of norms on $\R^k$}

\label{sec:t-norm}

\subsection{Freeness}\label{subsec:freeness}

A \emph{$*$-non-commutative probability space} is
a unital $*$-algebra $\mathcal A$ endowed with a tracial 
state $\phi$, i.e. a linear map $\phi\colon\mathcal A\to\mathbb C$
satisfying $\phi (ab)=\phi (ba),\phi (aa^{*})\geq 0, \phi (1)=1$.
An element of $\mathcal A$ is called
a (non-commutative) random variable. 
Let $\mathcal A_1, \ldots ,\mathcal A_k$ be subalgebras of $\mathcal A$ having the same unit as $\mathcal A$.
They are said to be \emph{free} if for all $a_i\in \mathcal  A_{j_i}$ ($i=1, \ldots, k$) 
such that $\phi(a_i)=0$, one has  
$$\phi(a_1\cdots a_k)=0$$
as soon as $j_1\neq j_2$, $j_2\neq j_3,\ldots ,j_{k-1}\neq j_k$.
Collections $S_{1},S_{2},\ldots $ of random variables are said to be 
free if the unital subalgebras they generate are free.

Let $(a_1,\ldots ,a_k)$ be a $k$-tuple of self-adjoint random variables and let
$\mathbb{C}\langle X_1 , \ldots , X_k \rangle$ be the
free $*$-algebra of noncommutative polynomials on $\mathbb{C}$ generated by
the $k$ self-adjoint indeterminates $X_1, \ldots ,X_k$. 
The {\it joint distribution\it} of the family $\{a_i\}_{i=1}^k$ is the linear form
\begin{align*}
\mu_{(a_1,\ldots ,a_k)} : \C\langle X_1, \ldots ,X_k \rangle &\to \C \\
P &\mapsto \phi (P(a_1,\ldots ,a_k)).
\end{align*}

Given a $k$-tuple $(a_1,\ldots ,a_k)$ of free 
random variables such that the distribution of $a_i$ is $\mu_{a_i}$, the joint distribution
$\mu_{(a_1,\ldots ,a_k)}$ is uniquely determined by the
$\mu_{a_i}$'s.
In particular, $\mu_{a_1+a_2}$ and $\mu_{a_1a_2}$ depend only on
$\mu_{a_1}$ and $\mu_{a_2}$. The notations $\mu_{a_1+a_2}=
\mu_{a_1}\boxplus\mu_{a_2}$ and $\mu_{a_1a_2}=\mu_{a_1}\boxtimes
\mu_{a_2}$ were introduced in Voiculescu's works \cite{V-boxplus,V-boxtimes}; operations $\boxplus$ and $\boxtimes$ 
are called the {\em free additive}, respectively
{\em free multiplicative} convolution. 
A family $(a_1^{n},\ldots ,a_k^{n})_n$ of $k$-tuples of random
variables is said to \emph{converge in distribution} towards $(a_1,\ldots ,a_k)$
iff for all $P\in \C \langle X_1, \ldots ,X_k \rangle$, 
$\mu_{(a_1^n,\ldots ,a_k^n)}(P)$ converges towards
$\mu_{(a_1,\ldots ,a_k)}(P)$ as $n\to\infty$. 
Sequences of random variables  $(a_1^{n})_n,\ldots ,(a_k^{n})_n$ are called \emph{asymptotically free} as $n \to \iy$
iff the $k$-tuple $(a_1^{n},\ldots ,a_k^{n})_n$ converges in distribution towards a family of free random variables.

The following result was contained in \cite{voiculescu-dykema-nica} (see also \cite{collins-sniady}).

\begin{theorem}\label{libre}
Let $\{U^{(n)}_k\}_{k \in \N}$ be a collection of independent 
Haar distributed random matrices of $\M_n (\C )$ and $\{W^{(n)}_k\}_{k\in \N}$ be a 
set of constant matrices of $\M_n (\C )$ 
admitting a joint limit distribution as $n \to \iy$ with respect to the
state $\phi_n = n^{-1}\trace$.
Then, almost surely,
 the family $\{U^{(n)}_k, W^{(n)}_k\}_{k \in \N}$ admits a limit $*$-distribution $\{u_k, w_k\}_{k \in \N}$ with respect to $\phi_n$, such that $u_1$, $u_2$, \ldots, $\{w_1, w_2, \ldots\}$ are free.
\end{theorem}

\subsection{Analytic transforms associated to free convolutions: definitions and reminders of classical results in complex analysis}
 
We start with the following classical definitions:

\begin{enumerate}[{I)}]
\item The \emph{Cauchy-Stieltjes transform} (or Cauchy transform) of a finite measure
$\mu$ on the real line:
$$
G_\mu(z)=\int_\mathbb R\frac{1}{z-t}\,d\mu(t),\quad
z\in\mathbb C\setminus\textrm{supp}(\mu),
$$
where $\textrm{supp}(\mu)$ denotes the topological support of $
\mu$. If $\mu$ is a positive measure, then $G_\mu$ maps the upper half
into the lower half of the complex plane, and $G_\mu(\overline{z})=
\overline{G_\mu(z)}$. Moreover, $\mu(\mathbb R)=\lim_{y\to+\infty}
iyG_\mu(iy)$.
\item $F_\mu(z)=1/G_\mu(z)$, $z\in\mathbb C\setminus\textrm{supp}(\mu).$
If the positive measure $\mu$ has compact support, then there exists a 
unique positive measure $\rho$ on the real line, whose support
is included in the convex hull of $\textrm{supp}(\mu)$ so that
$$
F_\mu(z)=\frac{z}{\mu(\mathbb R)}-\frac{\int_\mathbb Rt\,d\mu(t)}{
(\mu(\mathbb R))^2}+\int_\mathbb R\frac{1}{t-z}\,d\rho(t).
$$
This is a particular case of the so-called \emph{Nevanlinna representation}
of $F_\mu$ \cite[Equation 3.3]{akhieser}.
We shall almost exclusively be concerned with the case when 
$\mu(\mathbb R)=1$ and $\textrm{supp}(\mu)$ is a compact 
subset of $[0,+\infty)$. In that case, the total mass of $\rho$
equals the variance $\textrm{VAR}(\mu)$ of $\mu$:
 $\rho(\mathbb R)=\int s^2\,d\mu(s)-\left(\int s\,d\mu(s)\right)^2.$
\item The \emph{moment generating function} of a probability $\mu$
supported in $[0,+\infty)$ is
$$
\psi_\mu(z)=\int_{[0,+\infty)}\frac{zt}{1-zt}\,d\mu(t),\quad z\in
\mathbb C\setminus(1/\textrm{supp}(\mu)).
$$
It maps upper and lower half-planes into themselves.
It will be useful to note
\begin{equation}\label{eta-F&psi-G}
\psi_\mu(z)=\frac1z
G_\mu\left(\frac1z\right)-1,\quad \psi_\mu(0)=0.
\end{equation}

\item To compute free multiplicative convolutions of probability
distributions on $[0,+\infty)$ Voiculescu introduced the $S$-transform.
It is defined on a small enough neighborhood of zero as 
$$
S_\mu(z)=\frac{1+z}{z}\psi_\mu^{-1}(z),
$$
whenever $\mu\neq\delta_0$ is a compactly supported probability measure 
on $[0,+\infty)$. It satisfies the equation
\begin{equation}\label{S}
S_{\mu\boxtimes\nu}(z)=S_\mu(z)S_\nu(z) \quad \text{for} \quad |z|\textrm{ small}.
\end{equation}
From now on, unless otherwise specified, whenever we refer to 
$\psi_\mu^{-1}$, we refer to the inverse of $\psi_\mu$ around zero
and to its analytic continuation along the real line.
It is of interest to us to give a better description of 
the domain of injectivity of $\psi_\mu$ and the image of this
domain. A direct computation (see also \cite{BVIUMJ}) shows that
$\Im \psi_\mu'(z)>0$ for any $z$ in the upper half-plane 
for which $\Re z\le1/[\mu]$, where the notation $[\mu]$ is introduced 
in \eqref{def-bracket}. Since $\psi_\mu(\overline{z})=
\overline{\psi_\mu(z)}$ and $\psi_\mu$ preserves upper and lower 
half-planes, we conclude that $\psi_\mu$ is injective on 
$\{z\in\mathbb C\colon\Re z\le1/[\mu]\}$. On the other hand,
$\psi_\mu(x+iy)=\int\frac{t(x+iy)}{1-t(x+iy)}\,d\mu(t)=
\int\frac{tx-t^2(x^2+y^2))}{t^2y^2+(1-tx)^2}\,d\mu(t)+
iy\int\frac{t}{t^2y^2+(1-tx)^2}\,d\mu(t)$. We easily observe that 
$\Im\psi_\mu(x+iy)>\frac{y}{2(x^2+y^2+1)}\int\frac{t}{1+t^2}\,d\mu(t)$,
and, in particular,
$\Im\psi_\mu(x+i)>\frac{1}{2(x^2+1)}\int
\frac{t}{t^2+1}\,d\mu(t)$ for all $x\in\mathbb R$.
This gives us a bound on the ''thinness`` of the
domain of $\psi_\mu^{-1}$ in terms of the 
integral $\int
\frac{t}{t^2+1}\,d\mu(t)$. 

\end{enumerate}
These transforms have properties that make them important
in the study of free convolutions.

\medskip

Finally, we recall for the convenience of the reader a few classical results of complex analysis that we will need in the forthcoming proofs. 

\begin{enumerate}[{I)}]
\item
 The unit disc in the complex plane (and any conformally equivalent
domain) can be made into a metric space with a natural 
metric (the so-called pseudohyperbolic metric) with respect to which
any analytic self-map of the unit disc becomes a contraction. This is 
essentially the \emph{Schwarz-Pick Lemma}, which we
formulate here for the upper half-plane: If $f$ is an analytic 
self-map of the upper half-plane, then 
$$
\left|\frac{f(z)-f(w)}{f(z)-\overline{f(w)}}\right|\leq
\left|\frac{z-w}{z-\overline{w}}\right|, \quad \Im z,\Im w>0.
$$
Equality holds at a given pair of points if and only if $f$ is a
M\"{o}bius map.

In addition, if $z_0$ is a fixed point of $f$, and $f$ is not the
identity mapping or a rotation, 
then $z_0$ is the {\em unique} fixed point of $f$ and 
$|f'(z_0)|<1.$ The reader can find a wonderful presentation of this
subject (and much more) in the first chapter of \cite{Garnett}.

\item
There are self-maps of the upper half-plane that have no
fixed points in their domains. However, one can generalize this 
notion so that {\em all} such maps have a fixed point. In order
to do this, we should define the notion of \emph{non-tangential limit}.
The function $f$ defined on the upper half-plane has a non-tangential
limit $d$ at the point $x\in\mathbb R\cup\{\infty\}$ 
(and we shall write that as $\sphericalangle\lim_{z\to x}f(z)=d$)
if the limit of $f(z)$ exists and equals $d$ whenever $z$ 
approaches $x$ inside any closed cone $\Gamma$
included in $\{x\}\cup\{z\in\mathbb C\colon\Im z>0\}$.
This way one can also extend the notion of derivative: the 
\emph{Julia-Carath\'eodory derivative} of $f$ at a point $x\in\mathbb R$
where $\sphericalangle\lim_{z\to x}f(z)=d\in\mathbb R$
is defined as
$$
f'(x)=\sphericalangle\lim_{z\to x}\frac{f(z)-d}{z-x}.
$$
Remarkably, when the Julia-Carath\'eodory derivative of the
function $f$ is finite,
then $f'(x)=\sphericalangle\lim_{z\to x}f'(z)$.
If $x=d=\infty$, then the correct definition of the 
Julia-Carath\'eodory derivative  is 
$\sphericalangle\lim_{z\to\infty}\frac{z}{f(z)}$.
It is known that $f'(x)\in(0,+\infty]$.
It turns out that there can be infinitely many points $d\in
\mathbb R\cup\{\infty\}$ so that 
$\sphericalangle\lim_{z\to d}f(z)=d$. But if $f$ has no 
fixed point in the upper half-plane and is not a M\"obius map,
then there exists {\em exactly one }point $d\in\mathbb R\cup\{
\infty\}$ so that 
$$
\sphericalangle\lim_{z\to d}f(z)=d\quad {\rm and}\quad
f'(d)\in(0,1].
$$
A complete and very accessible reference for these results is \cite{Shapiro}.

\item

Non-tangential limits of an analytic map $f$ on the upper half-plane
can be said to uniquely determine $f$. Indeed, according to
a theorem due to Privalov, if there exists a set $E\subset\mathbb R$
of non-zero Lebesgue measure so that $\sphericalangle\lim_{z\to x}f(z)
=0$ for all $x\in E$, then $f$ is identically equal to zero \cite[Theorem 8.1]{CollingwoodL}.

\item
Conveniently, atoms of a probability measure $\mu$ can be easily 
expressed in terms of the Julia-Carath\'eodory derivatives of 
$F_\mu$ and $\psi_\mu$ as
$$
\begin{array}{cc}
\sphericalangle\lim_{z\to d}F_\mu(z)=0, & F_\mu'(d)=\alpha\\
\sphericalangle\lim_{z\to1/d}\frac{\psi_\mu(z)}{1+\psi_\mu(z)}=1, & \left(\frac{\psi_\mu(z)}{1+\psi_\mu(z)}\right)'(1/d)=d\alpha
\end{array}
$$
if and only if $\mu(\{d\})=1/\alpha$. In particular, if $d$ is an 
isolated atom of $\mu$, then both $F_\mu$ and $\psi_\mu/(1+\psi_\mu)
$ 
extend analytically around $d$.

\end{enumerate}

To conclude, let us note that if $\mu$ is the distribution 
of the self-adjoint random variable $y \in \mathcal A$ with respect to $\phi$,
then 
$$
G_\mu(z)=\phi\left((z-y)^{-1}\right),\quad z\not\in\sigma(y).
$$
This will be important in our study of norms of operators via
transforms. It follows from the above equality that
$\|y\|=\max\{\sup\textrm{supp}(\mu),-\inf\textrm{supp}(\mu)\}$,
so that $\|y\|$ can be described also as the maximum between
the largest $x\in\mathbb R$ in which $G_\mu$ is not analytic and
minus the smallest $x\in\mathbb R$ in which $G_\mu$ is not analytic.
If $y$ is a positive operator, then  
$\|y\|=\sup\textrm{supp}(\mu)$ and this number coincides 
with the largest $x\in\mathbb R$ in which $G_\mu$ is not analytic.

In terms of the transforms $F$ and $\psi/(1+\psi)$, we have the following
characterizations of $\|y\|$:
$$
\|y\|=\max(\{x\in\mathbb R\colon F_\mu(x)=0\}\cup\{x\in\mathbb R
\colon F_\mu\textrm{ not analytic in }x\}),
$$
and
\begin{eqnarray*}
\lefteqn{\|y\|^{-1}=}\\
& & \min(\{x\in\mathbb R\colon\psi_\mu(x)(1+\psi_\mu(x))^{-1}=1\}\cup
\{x\in\mathbb R\colon \psi_\mu(\cdot)(1+\psi_\mu(\cdot))^{-1}\textrm{ not analytic in }x\}).
\end{eqnarray*}
We shall denote
\begin{equation}
\label{def-bracket}
[\mu]:=\max \{|v| \colon v\in\textrm{supp}(\mu)\}.
\end{equation}

\subsection{The $(t)$-norm: definition}
\label{subsec:t-norm}

We introduce now a norm on $\R^k$ which will have a very important role to play in the description of the set $K_{n,k,t}$ in the asymptotic limit $n \to \iy$. 

\begin{definition}\label{def:t-norm}
For a positive integer $k$, embed $\R^k$ as a self-adjoint real 
subalgebra $\mathcal R$ of a $\mathrm{II}_1$ factor  $\mathcal A$ endowed with trace $\phi,$ 
so that $\phi((x_1,\dots,x_k))=(x_1+\cdots+x_k)/k$. Let $p_t$ be a projection of rank $t \in (0,1]$ in $\mathcal A$, free from $\mathcal R$. On the real vector space $\R^k$, we introduce the following norm, called the \emph{$(t)$-norm}:
\begin{equation}
	\normt{x}:=\norm{p_t x p_t}_{\infty},
\end{equation}
where the vector $x \in \R^k$ is identified with its image in $\mathcal R$.
\end{definition}

The fact that $\normt{\cdot}$ is indeed a norm deserves a proof, that we postpone to Lemma \ref{lem:properties-t-norm-without-R-S-transform}
in the next subsection.
Before that, we show that complex analysis stands as a powerful tool to study the distribution of $p_{t}xp_{t}$, and therefore of the $(t)$-norm.

Note that the distribution of the random variables $x$ and $p_t$ are, respectively $\mu_x = k^{-1}\sum_{i=1}^k \delta_{x_i}$ and $\mu_{p_t} = (1-t) \delta_0 + t \delta_1$. 
Therefore, in the framework of free probability and following the notation of Equation \eqref{def-bracket}, $\normt{x}=[\mu_{x}\boxtimes\mu_{p_{t}}]$ (recall definitions of operations $\boxplus$ and $\boxtimes$ from Section \ref{subsec:freeness}).

In the next proposition, we provide a free probabilistic description of the $(t)$-norm, which will turn out to be very useful. This result, first proved in \cite{NS-mult}, is contained in \cite{NS}, Lecture 14. 

\begin{proposition}\label{prop:t-norm-additive-conv}
The distribution $\mu_{t^{-1}p_txp_t}$ of the (non-commutative) random variable $t^{-1}p_txp_t$ in the $\mathrm{II}_1$ factor \emph{reduced} by the projection $p_t$ is related to the distribution $\mu_x$ of $x$ in the \emph{non-reduced} factor by the equation
\begin{equation}
	\mu_{t^{-1}p_txp_t}=\mu_x^{\boxplus 1/t},\quad t\in(0,1],
\end{equation}
where $\boxplus$ denotes the free additive convolution of Voiculescu. Hence, $\normt x$ is $t$ times the maximum between the upper bound and minus the lower bound of the support of the probability measure $\mu_x^{\boxplus 1/t}$.
\end{proposition}

It is possible to express the distribution of $p_txp_t$ in terms of the distribution of $x$, after the method described in \cite{BB-mathz,BB-imrn}: 
\begin{proposition}
Denoting $G_\mu(z)=\int_\mathbb R(z-t)^{-1}\,d\mu(t)$ the Cauchy-Stieltjes transform of a measure $\mu$ and $F_\mu(z)=1/G_\mu(z)$, the following relations hold
\begin{equation}\label{functional}
F_{\mu_x^{\boxplus 1/t}}(z)=F_{\mu_x}(\omega_{1/t}(z)),\quad
\omega_{1/t}(z)=tz+(1-t)F_{\mu_x^{\boxplus 1/t}}(z),
\end{equation}
so that the function $\omega_{1/t}$ is the right inverse of the function $H_{1/t}(w)=\frac1tw+\left(1-\frac1t\right)F_{\mu_x}(w)$, for $\Im w>0$. Moreover, $\omega_{1/t}$ extends continuously to the closure
of the upper half-plane.
\end{proposition}

\subsection{The $(t)$-norm: first properties}

We first prove properties about the $(t)$-norm that do not rely on complex analytic tools. 

\begin{lemma}\label{lem:properties-t-norm-without-R-S-transform}
The map $x\to \normt{x}$ defines indeed a norm. 
The $(t)$-norm $\normt{\cdot}$ has the following properties:
	\begin{enumerate}
	\item It is invariant under permutation of coordinates
		\begin{equation*}
			\normt{(x_1, x_2, \ldots, x_k)} = \normt{(x_{\sigma(1)}, x_{\sigma(2)}, \ldots, x_{\sigma(k)})} \qquad \forall \sigma \in S_k.
		\end{equation*}
	
		\item For all $s\geq 0$ (resp. $s\leq 0$) and for all vectors $x$ for which $\normt{x}$ is achieved at the {\em upper} 
		(resp. {\em lower})
		bound of the support of 
$\mu_x^{\boxplus 1/t}$,
		\begin{equation*}
			\normt{x + s(1^k)}=\normt{x} + s,
		\end{equation*}	
		(resp. $\normt{x - s(1^k)}=\normt{x} + s$).
\item The $(t)$-norm is determined by its restriction to the ordered probability simplex $\Delta_k^\downarrow$. 
\item
Whenever $t>1-\frac1k$ we have $\normt{x} = \norm{x}_\iy$

	\end{enumerate}
\end{lemma}
\begin{proof}
The fact that $\normt{\lambda x}=|\lambda |\normt{x}$ follows from the definition.
The triangle inequality follows from 
$$\norm{p_t (x+y) p_t}_{\infty}=
\norm{p_t x p_t+p_t y p_t}_{\infty}\leq \norm{p_t x p_t}+\norm{p_t y p_t}=\normt{x}+\normt{y}.$$
Now, assume that
$\normt{x}=0$. This is equivalent to $p_{t}xp_{t}=0$. In turn, this is equivalent to $p_{t}xp_{t}xp_{t}=0$, because
$x=x^{*}$. This is equivalent to $\phi( p_{t}xp_{t}xp_{t})=0$ because $\phi$ 
is faithful and $p_{t}xp_{t}xp_{t}$ is positive. 
But a direct computation shows that 
$\phi( p_{t}xp_{t}xp_{t})=t(t\phi (x^{2})+(1-t)\phi (x)^{2})$. Since $t\in [0,1]$, this can be zero iff $t=0$ or $x=0$.

The invariance under permutation follows from the fact that the moments of $p_{t}xp_{t}$ are symmetric functions in 
$x_{i}$, so this proves point (1).

Point (2) follows from the fact that
$p_{t}(x+s1)p_{t}= p_{t}xp_{t}+sp_{t}$ and from functional calculus.

The third point is a direct consequence of the second one. 

Writing $x=a_1q_1+\cdots+a_kq_k$, $x$ reaches its norm $a_k$ on a
projection $q_k$ of trace at least $1/k$, so $\phi(\inf\{p_t,q_k\})\ge
\phi(p_t)+\phi(q_k)-1>0$, and hence 
$p_txp_t\ge a_k\inf\{p_t,q_k\}$, so $\|p_txp_t\|=\|x\|_\infty=a_k$.
\end{proof}
It might be worth 
noting that with a little extra effort one can show that the equality 
$\|x\|_{(t)}=\|x\|_\infty$ from the above proposition holds also when 
$t=1-\frac1k$.

In general it is difficult to explicitly compute the $(t)$-norm. We gather in the next proposition some important properties that 
can be obtained with methods of complex analysis.

\begin{proposition}\label{53} The $(t)$-norm $\normt{\cdot}$ has the following properties:
	\begin{enumerate}
		\item[1.] For any $x\in\mathbb R^k$,
		\begin{equation}
			\frac1t\normt{x} = \frac1tw_x+\left(1-\frac1t\right)F_{\mu_x}(w_x),
		\end{equation}	
where $w_x$ is the largest in absolute value solution to the equation
\begin{equation}
F_{\mu_x}(w)\left(F_{\mu_x}'(w)-\frac{1}{1-t}\right)=0.
\end{equation}
Moreover, the map $t\mapsto\|x\|_{(t)}$ is non-decreasing on 
$(0,1]$.
	\item[2.] For all $j=1, 2, \ldots, k$, one has
		\begin{equation}\label{eq:norm2proj}
			\normt{(1^j0^{k-j})}=
			\begin{cases}
				t +u-2tu +2\sqrt{tu (1-t)(1-u)}  &\text{ if } t + u < 1,\\
				1  &\text{ if } t + u \geq 1,
			\end{cases}
		\end{equation}
		where $u = j/k$.
	\end{enumerate}
\end{proposition}

\begin{proof}

As it is more natural in probabilistic terms to do it, we shall 
make the change of parameter $s=1/t$. 
Note that in terms of probability measures, $\mu_x$ is
purely atomic and compactly supported, hence $G_{\mu_x}$ is 
a rational function analytic on a neighbourhood of infinity which
maps $\mathbb R\cup\{\infty\}$ into itself.
Moreover, the radius of convergence around infinity for $G_{\mu_x}$
equals $\|x\|_\infty$ (in the sense that $G_{\mu_x}(z)=
\sum_{n=0}^\infty\left(\int t^nd\mu_x(t)\right)z^{-n-1}$ for 
$|z|>\|x\|_\infty$). It follows that $F_{\mu_x}$ is also a 
rational function which maps  $\mathbb R\cup\{\infty\}$ into itself.
Moreover the Nevanlinna representation \cite[Equation 3.3]{akhieser}
of $F_{\mu_x}$  reads
$$
F_{\mu_x}(z)=a+z+\int_\mathbb R\frac{1}{t-z}\,d\rho(t),
$$
where $a=-\int t\,d\mu_x(t)$ and $\rho$ is a compactly supported
purely atomic positive measure on the real line with total mass
$\rho(\mathbb R)=\textrm{VAR}(\mu_x)$. A direct
computation shows that $\|x\|_\infty$ is the largest, in absolute
value, solution of the equation $F_{\mu_x}(v)=0$, and, moreover,
$F_{\mu_x}(z)-z$ is analytic around infinity, with a radius of
convergence {\em strictly} greater than the radius of convergence
of $G_{\mu_x}$ (in the sense that $F_{\mu_x}(z)-z$ is analytic
on the complement of a disc of radius strictly smaller than 
the one corresponding to $G_{\mu_x}$). This last statement is 
clearly true for any probability $\mu$ for which
$[\mu] = \max\{|v|\colon v\in\textrm{supp}(\mu)\}$ is reached at an 
isolated atom of $\mu$.

Thus, as $\|p_txp_t\|=\frac1s\max\{|a|\colon a\in{\mathrm{supp}}(\mu_x^{\boxplus s})\},$ it follows that $\normt{x}/t$ coincides with the largest in absolute value real number $v$ so that either $F_{\mu_x^{\boxplus s}}(v)=0$ or $F_{\mu_x^{\boxplus s}}$ is not analytic in $v$,
with the first case corresponding to the situation in which the
maximum is reached at an isolated atom of $\mu_x^{\boxplus s}$.
The first statement follows from the above observation and from 
 the Definition \ref{def:t-norm} and the Proposition \ref{prop:t-norm-additive-conv}.

Denote $J$ the interval in $\mathbb R$ containing arbitrarily
large positive numbers on which $F_{\mu_x}$ is analytic; clearly,
$J\supseteq(\|x\|_\infty,+\infty)$. Also, denote $J_s$ the 
similar interval corresponding to $F_{\mu_x^{\boxplus s}}$.
From the Nevanlinna representation, we  gather the following:
\begin{itemize}
\item For any $s\ge 1$, $z\in J_s$, 
$$F_{\mu_x^{\boxplus s}}(z)\leq z-\int t\,d\mu_x^{\boxplus s}(t),\quad F_{\mu_x^{\boxplus s}}'(z)>1,
\quad F_{\mu_x^{\boxplus s}}''(z)<0.$$
\item 
If $(\mu_x^{\boxplus s})^{\rm ac}$ denotes the (necessarily
non-zero whenever $s>1$) absolutely continuous part of $\mu_x^{\boxplus s}$,
then 
\begin{eqnarray*}
\inf J_s & = & 
\max\{v\colon v\in{\mathrm{supp}}(\mu_x^{\boxplus s})^{\rm ac}\}\\
& = &
\max\{v\colon F_{\mu_x^{\boxplus s}}\textrm{ not analytic in }v\}\\
& = &
\max\{v\colon \omega_s\textrm{ not analytic in }v\};
\end{eqnarray*}
\item Let us denote $x(s)=\inf J_s$. Then 
$$
x(s)=sv(s)+(1-s)F_{\mu_x}(v(s)),\quad s\ge1,
$$
where $v(s)$ is the largest solution of the equation $F_{\mu_x}'
(v)=\frac{s}{s-1}$.
\end{itemize}
Only the last item needs some justification: it follows from 
equation \eqref{functional} that the domains of analyticity of
$\omega_s$ and $F_{\mu_x^{\boxplus s}}$ coincide.
Moreover, $\omega_s$ being the right inverse of $H_s$,
it follows that $H_s(\omega_s(z))=z$ for all $z\in J_s$
and $\omega_s(H_s(z))=z$ for all $z\in H_s(J_s)$. 
Computing the derivative $H_s'(z)=s+(1-s)F_{\mu_x}'(z)$
and using the first item above, it follows  that 
the first obstacle for the analytic extension of $\omega_s$
along $\mathbb R$ coming from $+\infty$ is the point $H_s(v(s))$
with $v(s)$ described in the last item above. Then,
$x(s)=H_s(v(s))=sv(s)+(1-s)F_{\mu_x}(v(s)).$

Elementary implicit differentiation gives 
$$
x'(s)=v(s)+sv'(s)-F_{\mu_x}(v(s))+(1-s)F_{\mu_x}'(v(s))
v'(s)=v(s)-F_{\mu_x}(v(s)).$$
We have used above the fact that $F_{\mu_x}'
(v(s))=\frac{s}{s-1}$. Then

\begin{eqnarray*}
\partial_s\left(\frac{x(s)}{s}\right)& = & 
\frac{sx'(s)-x(s)}{s^2}\\
& = &
\frac{sv(s)-sF_{\mu_x}(v(s))-sv(s)-(1-s)F_{\mu_x}(v(s))}{s^2}\\
& = & 
-\frac{F_{\mu_x}(v(s))}{s^2}.
\end{eqnarray*}
As noted, if $\normt{x}$ is achieved at the upper bound of the 
support of the distribution of $p_txp_t$, then $\normt{x}=
\frac{x(s)}{s}$ whenever $\normt{x}$ is not achieved at an atom
of $\mu_x^{\boxplus s}$.

To complete the proof, we observe that without loss of generality, we 
may assume that $\normt{x}$ is achieved at the upper bound of the 
support of our measure. If this upper bound coincides with an 
atom of the measure, then we have already seen in Lemma
\ref{lem:properties-t-norm-without-R-S-transform} that 
$\normt{x}=\|x\|_\infty.$ If that is not the case, then
$\normt{x}=\frac{x(s)}{s}$. We claim that $F_{\mu_x}(v(s))\ge0$.
Indeed, if $F_{\mu_x}(v(s))<0$, then there must be some point 
$z_0>v(s)$ so that $F_{\mu_x}(z_0)=0$ and hence $H_s(z_0)=sz_0$. 
But $v(s)=\omega_s(x(s))$, $\omega_s$ is defined right of $x(s)$
and $sz_0>x(s)$, hence 
$z_0=\omega_s(sz_0)=\frac1s(sz_0)+\left(1-\frac1s\right)F_{
\mu_x^{\boxplus s}}(sz_0)$ implies 
$F_{\mu_x^{\boxplus s}}(sz_0)=0$, so $sz_0$ is an atom for 
$\mu_x^{\boxplus s}$, a contradiction. Thus, the function
$s\mapsto \normt{x}$ is non-increasing, strictly decreasing
when $\normt{x}$ is reached at the boundary of the support of the
absolutely continuous part of $\mu_x^{\boxplus s}$.

Note that our proof does not exclude the possibility that,
as $t$ decreases, $\normt{x}$ could switch from being achieved
at the upper bound of the support of $\mu_x^{\boxplus s}$
to being achieved at its lower bound. However, the argument 
above still holds even if such a switch happens.

 For the last item, see \cite{voiculescu-dykema-nica}, example 3.6.7. 
 This is one of the few cases when an exact expression for the $(t)$-norm is known and it has been heavily used in \cite{collins-nechita-2}.
\end{proof}

In Figure \ref{fig:k2}, the ball for the $(t)$-norm is plotted for $k=2$. Note that the shape of the ball depends only on the parameter 
$$
x_t=\begin{cases}
\frac{1}{2\sqrt{t(1-t)}}, \quad &\text{ if } t < \frac{1}{2},\\
1, \quad &\text{ if } t \geq \frac{1}{2}
\end{cases}
$$
whose dependence in $t$ is also plotted in the right-hand side subfigure.
\begin{figure}[htbp]
\centering
\subfigure[]{\includegraphics[width=6cm]{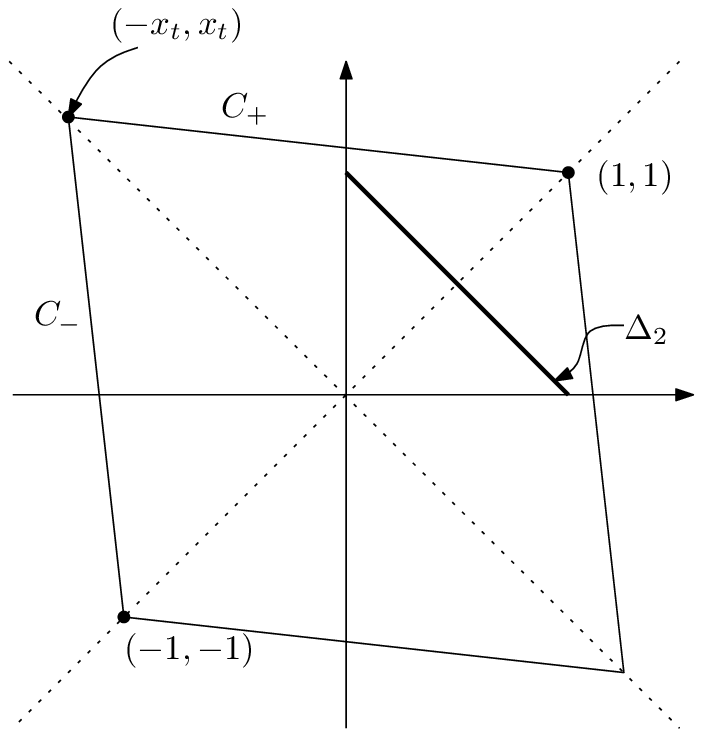}}\qquad\qquad
\subfigure[]{\includegraphics[width=6cm]{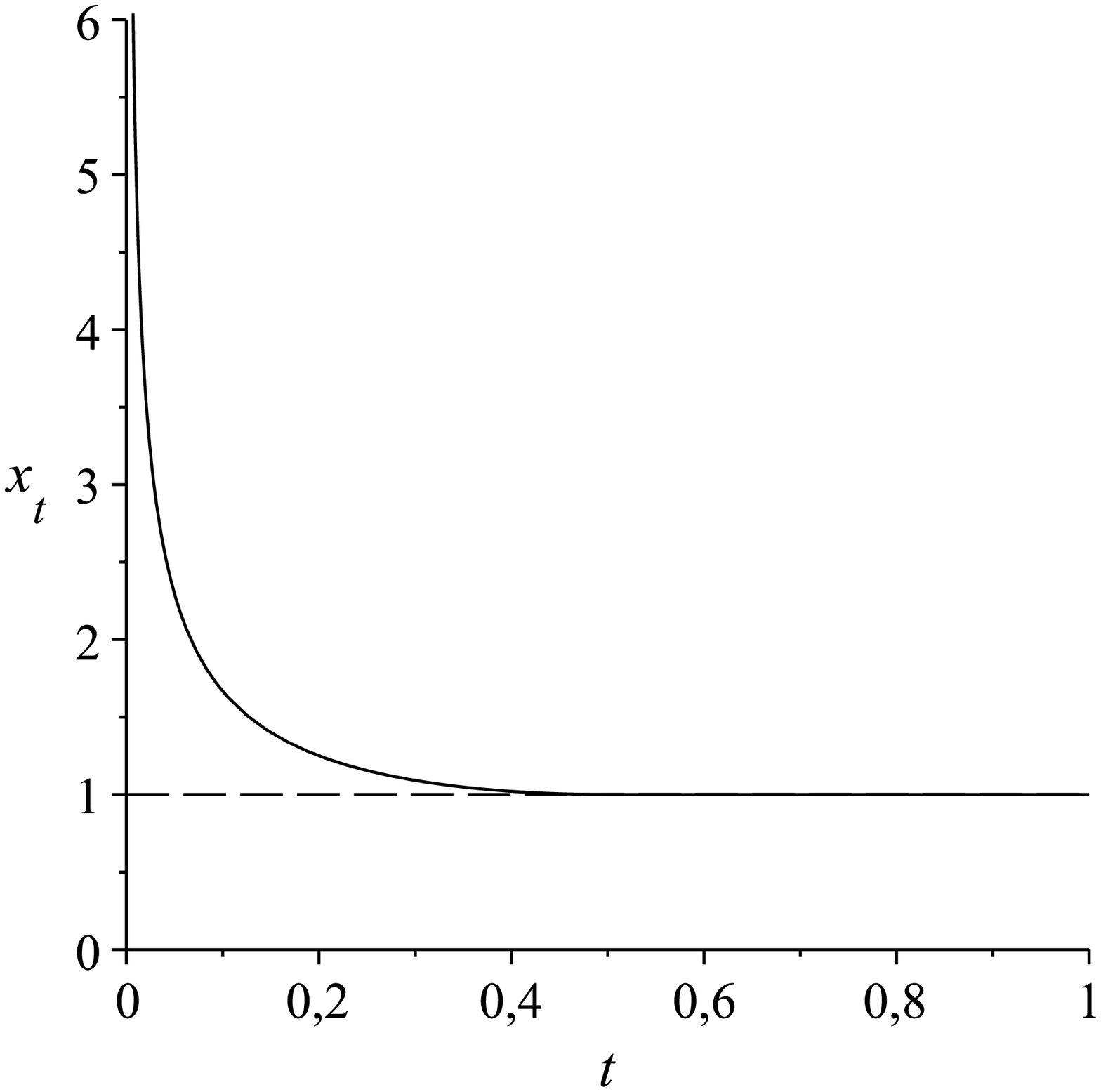}}\\
\caption{The unit ball for the $(t)$-norm in $\R^2$.}
\label{fig:k2}
\end{figure}

Let us mention that the solution to the equation $F_{\mu_x}(w)=0$ corresponds to an atom, that is,
if the solution $w_x$ is of $F_{\mu_x}(w)=0$, the norm t is
achieved either at an atom of $\mu_x^{\boxplus t}$ or at a point where
the density of this measure is infinite. 
Atoms of the probability measure $\mu_x^{\boxplus 1/t}$ have been fully described in \cite{BB-mathz} by the formula
\begin{equation}\label{atoms}
\mu_x^{\boxplus 1/t}\left(\{a\}\right)=\max\left\{0,\frac1t\mu_x\left(\{ta\}\right)-\frac1t+1\right\}.
\end{equation}

Let us record for further use that the above implies that
when $t<\frac1k$ the measure $\mu_x^{\boxplus 1/t}$ is necessarily absolutely continuous with respect to the Lebesgue measure on $\mathbb R$.

\subsection{The $(t)$-norm: continuity}

This section contains a technical result for the continuity of the $(t)$-norm in $t$ and of the $\boxtimes$ operation. Proposition \ref{prop:boxtimes-cont} is the main result here and it has independent interest in free probability. In the rest of this paper, we shall use a simpler incarnation of this result in the form of Corollary \ref{cor:boxplus-cont-util}.

\begin{proposition}\label{prop:t-norm-cont}
Assume that $\mu$ is a compactly supported probability measure on $[0,
+\infty)$. Then the map $[1,+\infty)\ni t\mapsto[\mu^{\boxplus t}]\in
(0,+\infty)$ is continuous, algebraic outside a bounded discrete
subset of $(1,+\infty)$.
Moreover,
\begin{equation}\label{estimate}
0<[\mu^{\boxplus t+\varepsilon}]-[\mu^{\boxplus t}]<
t\left(\sqrt{\varepsilon\textrm{VAR}(\mu)}+
\varepsilon\textrm{VAR}(\mu)\right),\quad t\ge1, \varepsilon>0.
\end{equation}
\end{proposition}

\begin{proof}
As noted before, $[\mu^{\boxplus t}]$ is the largest positive number 
where either $F_{\mu^{\boxplus t}}$ is not analytic, or 
$F_{\mu^{\boxplus t}}$ takes the value zero.
We shall use equations 
\eqref{functional}
in order to analyze this number. It follows easily that 
$F_{\mu^{\boxplus t}}$ is not analytic in $x_0$ if and only if
$\omega_t$ is not analytic in $x_0$. This latter function is the
right inverse of
$$
H_t(w)=tw+(1-t)F_\mu(w)=w+(t-1)\int_{[0,+\infty)} s\,d\mu(s)+(1-t)
\int_{[0,+\infty)}\frac{1}{s-w}\,d\rho(s),
$$
according to the Nevanlinna representation of $F_\mu$.

One can see directly that for any $w$ in the interval
of analyticity of $H_t$ included in $([\mu],+\infty)$,
we have $H_t(w)>w$, $H_t'(w)=1+(1-t)
\int_{}\frac{1}{(s-w)^2}\,d\rho(s).$ For simplicity, we shall denote
$x_t$ the largest point in the real line in which $\omega_t$
is not analytic. Thus, $H_t$ maps the interval 
$[\max\{(H_t')^{-1}(\{0\}),[\rho]\},+\infty)$ bijectively onto 
$[x_t,+\infty)$. For $t>1$ large enough, it is clear that
$[\max\{(H_t')^{-1}(\{0\}),[\rho]\},+\infty)=
[\max(H_t')^{-1}(\{0\}),+\infty)$, and so the correspondence 
$t\mapsto\max(H_t')^{-1}(\{0\})$ is clearly algebraic (in fact
analytic). The relation $H_t(\omega_t(x))=x$ implies that
$x_t=H_t(\max(H_t')^{-1}(\{0\})),$ which is an analytic function.
As $t$ decreases towards 1, it may happen 
(whenever $\lim_{w\downarrow[\rho]}\int_{}\frac{1}{(s-w)^2}\,d\rho(s)<+
\infty$)
that $\max(H_t')^{-1}(\{0\})$ either does not exist, or is no greater
than $[\rho]$. We shall note that in this case there is a $t_0=
1+\left(\lim_{w\downarrow[\rho]}\int_{}\frac{1}{(s-w)^2}\,d\rho(s)\right)^{-1}$ so that the function $t\mapsto x_t$ is analytic
on $(t_0,+\infty)$ and extends continuously to $t_0$. On the 
interval $[1,t_0]$ we have, by the same relation $H_t(\omega_t(x))=x$,
$$
x_t=H_t([\rho])=t[\rho]+(1-t)\lim_{w\downarrow[\rho]}F_\mu(w),
$$
which is again an analytic (linear!) map of $t$. We note that 
$\lim_{w\downarrow[\rho]}F_\mu(w)$ must be finite as long as 
$\lim_{w\downarrow[\rho]}\int_{}\frac{1}{(s-w)^2}\,d\rho(s)<+\infty$.

This has determined the analyticity of the correspondence between
$t$ and the largest point of non-analyticity of $\omega_t$,
and hence of $F_{\mu^{\boxplus t}}$. We have however remarked at the beginning of our proof that this point does not necessarily coincide 
with $[\mu^{\boxplus t}]$, and that moreover, the case in which it does not 
coincide corresponds to the case of an isolated atom of 
$\mu^{\boxplus t}$. Atoms of $\mu^{\boxplus t}$ have been 
however described in equation \eqref{atoms}; it follows that
the correspondence remains linear for $t$ in the interval 
$[1,(1-\mu(\{a\}))^{-1}]$. Moreover, when 
$t=(1-\mu(\{a\}))^{-1}$, we have $H_t'(a)=t+(1-t)F_\mu'(a)=
t+(1-t)/\mu(\{a\})=(1-\mu(\{a\}))^{-1}+(1-(1-\mu(\{a\}))^{-1})/\mu(\{a\})=0$ (derivatives understood either in their proper sense, or in the 
Julia-Carath\'eodory sense), so at $t=(1-\mu(\{a\}))^{-1}$ we
encounter a breach of analyticity of $\omega_t$ at the point $ta$.

This allows us to conclude that $t\mapsto[\mu^{\boxplus t}]$ has two
possible regimes of evolution, either linear or according to 
$H_t(\max(H_t')^{-1}(\{0\}))$, and the two regimes ``glue''
continuously. This guarantees continuity of
$t\mapsto[\mu^{\boxplus t}]$ on $(0,+\infty)$. 
If the linear evolution occurs at all, then continuity
at $t=1$ is obvious. If it does not, then we observe that
$\lim_{t\to1}\max(H_t')^{-1}(\{0\})=[\rho]=[\mu]$,
and moreover we can specify, by the Nevanlinna representation, that
$$
0<\max(H_t')^{-1}(\{0\})-[\mu]<\sqrt{(t-1)\rho(\mathbb R)}=
\sqrt{(t-1)\left(\int s^2\,d\mu(s)-\left[\int s\,d\mu(s)\right]^2
\right)}.
$$
Then 
$$[\mu^{\boxplus t}]=
H_t(\max(H_t')^{-1}(\{0\}))=\max(H_t')^{-1}(\{0\})+
(t-1)\int s-\frac{1}{s-\max(H_t')^{-1}(\{0\})}\,
d\rho(s),
$$
so it is enough to estimate 
$\left|
(t-1)\int\frac{1}{s-\max(H_t')^{-1}(\{0\})}\,
d\rho(s)\right|.$ Recalling the equation determining 
$\max(H_t')^{-1}(\{0\})$, namely
$\int\frac{1}{(s-\max(H_t')^{-1}(\{0\}))^2}\,
d\rho(s)=\frac{1}{t-1}$, and noting that $\left(\int\frac{1}{s-\max(H_t')^{-1}(\{0\})}\,
d\rho(s)\right)^2<\rho(\mathbb R)\int\frac{1}{(s-\max(H_t')^{-1}(\{0\}))^2}\,
d\rho(s)$, we get
\begin{eqnarray*}
\left(
(t-1)\int\frac{1}{s-\max(H_t')^{-1}(\{0\})}\,
d\rho(s)\right)^2 & < & (t-1)^2
\int\frac{\rho(\mathbb R)}{(s-\max(H_t')^{-1}(\{0\}))^2}\,
d\rho(s)\\
&= & (t-1)\rho(\mathbb R).
\end{eqnarray*}
We obtain 
$$
0<[\mu^{\boxplus t}]-[\mu]<\sqrt{(t-1)\rho(\mathbb R)}+
(t-1)\rho(\mathbb R).
$$
This, together with the fact that the variance of $\mu^{\boxplus t}$
equals $t$ times the variance of $\mu$, guarantees that 
$$
0<[\mu^{\boxplus t+\varepsilon}]-[\mu^{\boxplus t}]<
t(\sqrt{\varepsilon\rho(\mathbb R)}+
\varepsilon\rho(\mathbb R)).
$$
Since $\rho(\mathbb R)=\textrm{VAR}(\mu)$, this concludes our proof.
\end{proof}

Note that, while the estimate provided by the above 
lemma is indeed optimal at $t=1$, it is not optimal
throughout $(1,+\infty)$. However, it will serve our purposes.
Also, it is worth mentioning that the correspondence $t\mapsto
[\mu^{\boxplus t}]$ may fail to be analytic on $(1,+\infty)$ only
due to a ``phase transition'' from a linear to an essentially
inverse quadratic regime.

Next we address the problem of continuity for the upper bound of the 
support of the multiplicative free convolution of two probability 
distributions on the positive half-line. More precisely, assume that
there is a topological space $X$ and a pair of functions 
$f,g\colon X\to (\mathcal A^+,\varphi)$, where $\mathcal A^+$ denotes
the set of positive elements in the non-commutative probability
space $(\mathcal A,\varphi)$. Assume that $f,g$ are weak* continuous
(meaning that $X\ni\xi\mapsto\mu_{f(\xi)}$ is continuous from the
topology of $X$ to the weak topology on the space of probability
distributions compactly supported on $[0,+\infty)$, and the same
for $g$), and in addition the maps $X\ni\xi\mapsto\|f(\xi)\|,
X\ni\xi\mapsto\|g(\xi)\|$ are continuous. As noted before,
$\|f(\xi)\|=[\mu_{f(\xi)}],$ and we shall use the two notations
interchangeably.

It was noted before that $[\mu_{f(\xi)}]$ coincides with 
$\max\{x\in\mathbb R\colon G_{\mu_{f(\xi)}}\textrm{ not analytic in
} x\}$. Equation 
\eqref{eta-F&psi-G} allows us to re-phrase
this in terms of the moment generating function as 
$$
\frac1{[\mu_{f(\xi)}]}=
\min\{x\in\mathbb R\colon \psi_{\mu_{f(\xi)}}\textrm{ not analytic in
} x\}.
$$
Let us recall that for any $\mu\neq\delta_0$ supported on the positive
half-line, $\psi_\mu$ is strictly increasing on the interval
$(-\infty,1/[\mu])$, so $\lim_{x\uparrow1/[\mu]}\psi_\mu(x)$ exists
in $(0,+\infty]$. We shall denote it by $\psi_\mu(1/[\mu])$. In 
particular, the inverse function $\psi^{-1}_\mu$ of $\psi_\mu$
is defined on $(\mu(\{0\})-1,\psi_\mu(1/[\mu])]$, monotonic,
and takes values in $\left(-\infty,1/[\mu]\right]$.
However, it is clear that $\psi_\mu^{-1}$ might have an analytic
extension beyond $\psi_\mu(1/[\mu])$; indeed, that would correspond
to the case when $(\psi_\mu^{-1})'(\psi_\mu(1/[\mu]))=0$.
Thus, we can give a description of $[\mu]$ in terms of
$\psi_\mu^{-1}$:
$$
\frac1{[\mu]}=\min\left\{\psi_\mu^{-1}(x)\colon(\psi_\mu^{-1})'(r)>0
\forall r<x,\psi_\mu^{-1}\textrm{ not analytic in }x \textrm{ or }(\psi_\mu^{-1})'(x)=0
 \right\}.
$$
(The case $x=+\infty$ is not excluded.)

The following proposition is concerned with the continuity of the
correspondence
$X\ni\xi\mapsto[\mu_{f(\xi)}\boxtimes\mu_{g(\xi)}]$ or, equivalently,
the correspondence
$X\ni\xi\mapsto\|f(\xi)g(\xi)\|$, where the sets $f(X)$ and $g(X)$ are
assumed to be free with respect to $\varphi$. For mere convenience,
we assume $X$ to be a metric space. We shall denote by
$M_{f(\xi)}$ the largest positive number with the property that 
$\psi_{\mu_{f(\xi)}}^{-1}$ extends analytically to a complex 
neighbourhood of the interval $(\mu_{f(\xi)}(\{0\})-1,M_{f(\xi)}).$
We will assume that $\psi_{\mu_{f(\xi)}}^{-1}$ extends continuously
as a real function to $M_{f(\xi)}$ and we will denote by
$\tilde{\psi}_{\mu_{f(\xi)}}^{-1}$ the continuous extension
$$\tilde{\psi}_{\mu_{f(\xi)}}^{-1}(x)=\left\{\begin{array}{lcl}
{\psi}_{\mu_{f(\xi)}}^{-1}(x) & \textrm{if} & x<M_{f(\xi)}\\
\displaystyle\lim_{r\uparrow M_{f(\xi)}}{\psi}_{\mu_{f(\xi)}}^{-1}(r)
& \textrm{if} & x\ge M_{f(\xi)}
\end{array}\right.
$$

\begin{proposition}\label{prop:boxtimes-cont}
Let $X$ be a metric space, $(\mathcal A,\varphi)$ a 
non-commutative probability space and 
$f,g\colon X\to (\mathcal A^+\setminus\{0\},\varphi)$ two norm-bounded 
functions that take values in free subalgebras of $\mathcal A$ 
satisfying the following conditions:
\begin{enumerate}
\item The correspondences $X\ni\xi\mapsto\mu_{f(\xi)},\mu_{g(\xi)}$
are weakly continuous;
\item The correspondences $X\ni\xi\mapsto\|f(\xi)\|,\|g(\xi)\|\in(0,+
\infty)$ are continuous;
\item The correspondences $X\ni\xi\mapsto M_{f(\xi)},M_{g(\xi)}
\in(0,+\infty]$ are continuous;
\item The correspondences $X\ni\xi\mapsto
\tilde{\psi}_{\mu_{f(\xi)}}^{-1},\tilde{\psi}_{\mu_{g(\xi)}}^{-1}$ 
are continuous in the uniform norm, in the sense that for any $\xi_0
\in X$,
$$
\lim_{\xi\to\xi_0}\sup_{r\in[0,+\infty]}|\tilde{\psi}_{\mu_{f(\xi)}}^{-1}(r)
-\tilde{\psi}_{\mu_{f(\xi_0)}}^{-1}(r)|=0.
$$
\item If $\psi_{\mu_{f(\xi_0)}}^{-1}$ is analytic 
on some complex neighborhood of $[0,r]$, then there exists a 
neighborhood $\mathcal U$ of $\xi_0$ and a complex
neighborhood $V$ of $[0,r]$ so that $\psi_{\mu_{f(\xi)}}^{-1}$
is analytic on $V$ for all $\xi\in\mathcal U$. Same statement is 
required to hold for $g$.
\end{enumerate}
Then the correspondence $X\ni\xi\mapsto[\mu_{f(\xi)}\boxtimes\mu_{g(\xi)}]\in(0,+\infty)$ is continuous.
\end{proposition}

Before starting the proof, we should mention that, as weak 
continuity for $f$ (condition (1)) is equivalent to continuity 
in the topology of the uniform convergence on compacts for $\psi_{
\mu_{f(\xi)}}$, if $\{\xi_n\}_{n\in\mathbb N}\subseteq X$ converges to 
$\xi_0$ and $\psi_{\mu_{f(\xi_n)}}^{-1},\psi_{\mu_{f(\xi_0)}}^{-1}$
have a common domain, then $\psi_{\mu_{f(\xi_n)}}^{-1}$ converges
to $\psi_{\mu_{f(\xi_0)}}^{-1}$ uniformly on compacts of the common
domain. Condition (5) is devised in order to efficiently 
exploit this property. Condition (4) is a bit stronger than it appears:
it says that if $\xi_n$ converges to $\xi$ in $X$ and $r_n\in[0,
M_{f(\xi_n)}]$ converges to $r\in[0,M_{f(\xi)}]$, then 
$\psi_{\mu_{f(\xi_n)}}^{-1}(r_n)$ converges to 
$\psi_{\mu_{f(\xi)}}^{-1}(r)$ as $n\to\infty$. Indeed, 
$|\psi_{\mu_{f(\xi)}}^{-1}(r)-\tilde{\psi}_{\mu_{f(\xi)}}^{-1}(r_n)|
\to0$ as $n\to\infty$ by the continuity of 
$\tilde{\psi}_{\mu_{f(\xi)}}^{-1}$, and 
$|\tilde{\psi}_{\mu_{f(\xi)}}^{-1}(r_n)-
{\psi}_{\mu_{f(\xi_n)}}^{-1}(r_n)|
\to0$ as $n\to\infty$ by condition (4). In addition, our convention
for doing arithmetics with infinity are $\infty+$
a real number $=\infty$, and $\infty -\infty=0$, so that 
the sets $\{x:\tilde{\psi}_{\mu_{f(\xi)}}^{-1}(x)=+\infty\}$
and $\{x:\tilde{\psi}_{\mu_{f(\xi_n)}}^{-1}(x)=+\infty\}$ must 
coincide when $n$ is large enough.

\begin{proof}
The statement of the proposition is local in nature: thus, let us choose 
$\xi_0\in X$ and an arbitrary sequence $\{\xi_n\}_{n\in\mathbb N}
\subseteq X$ converging to $\xi_0$.
It should be 
recorded that condition (1) and the weak continuity result of 
Bercovici and Voiculescu \cite{BVIUMJ} for free multiplicative 
convolution implies that 
$\liminf_{n\to\infty}[\mu_{f(\xi_n)}\boxtimes\mu_{g(\xi_n)}]\ge[\mu_{f(\xi_0)}\boxtimes\mu_{g(\xi_0)}]$ (we might ``lose,'' but not ``gain''
support when passing to weak limit).
We shall prove that $\lim_{n\to\infty}[\mu_{f(\xi_n)}\boxtimes\mu_{g(\xi_n)}]=[\mu_{f(\xi_0)}\boxtimes\mu_{g(\xi_0)}].$ In order to do that, we shall
use \eqref{S}, the $S$-transform property of Voiculescu. This
translates, in terms of the moment-generating function, in
$$
\psi_{\mu_{f(\xi)}\boxtimes\mu_{g(\xi)}}^{-1}(z)=\frac{1+z}{z}
\psi_{\mu_{f(\xi)}}^{-1}(z)\psi_{\mu_{g(\xi)}}^{-1}(z).
$$
This relation holds for $z$ in the interval bounded below by
$\max\left\{\mu_{f(\xi)}(\{0\}),\mu_{g(\xi)}(\{0\})\right\}-1$ and above
by the minimum between the domains of $\psi_{\mu_{f(\xi)}}^{-1}$ 
and $\psi_{\mu_{g(\xi)}}^{-1}$ viewed as inverses of the corresponding functions. However, there are many circumstances in which the above
equality can be continued analytically (as complex functions) further 
along the positive axis.
The maximum domain in $\mathbb R$ is the interval
$(\max\left\{\mu_{f(\xi)}(\{0\}),\mu_{g(\xi)}(\{0\})\right\}-1,M_\xi]$,
where $M_\xi$ is no smaller than the least of the upper bounds
$M_{f(\xi)},M_{g(\xi)}$ of the 
domains of $\psi_{\mu_{f(\xi)}}^{-1}$ and $\psi_{\mu_{g(\xi)}}^{-1}$.
In that case, $1/[\mu_{f(\xi)}\boxtimes\mu_{g(\xi)}]$ equals the 
either $\psi_{\mu_{f(\xi)}\boxtimes\mu_{g(\xi)}}^{-1}(M_\xi)$
or $\psi_{\mu_{f(\xi)}\boxtimes\mu_{g(\xi)}}^{-1}(x_\xi)$, where
$x_\xi$ is the smallest critical point of 
$\psi_{\mu_{f(\xi)}\boxtimes\mu_{g(\xi)}}^{-1}$, if existing.

For simplicity, we shall denote  $a_n=\psi_{\mu_{f(\xi_n)}}^{-1},
b_n=\psi_{\mu_{g(\xi_n)}}^{-1},c_n=
\psi_{\mu_{f(\xi_n)}\boxtimes\mu_{g(\xi_n)}}^{-1}$, with the obvious
changes when $n$ is replaced by $0$ or simply eliminated. 
We shall split the proof in two cases:
\newline\noindent {\bf Case 1:} There exists a point $x_{\xi_0}>0$ in
the domain of $c_0$ so that $c_0'(x_{\xi_0})=0$ as a complex function. Without loss
of generality, we may assume that this point $x_{\xi_0}$ is the
smallest satisfying this condition, so that $c_0(x_{\xi_0})=
1/[\mu_{f(\xi_0)}\boxtimes\mu_{g(\xi_0)}]$ (and thus $c_0$ extends 
analytically on a complex neighborhood of $[0,x_{\xi_0}]$).

Thus, by the $S$-transform property, there is a 
neighborhood of $[0,x_{\xi_0})$ on which $a_0$ and $b_0$ extend
analytically. Indeed, assume towards contradiction
there exists a point $r\in(0,
x_{\xi_0})$ so that, say, $a_0$ does not extend analytically to
it. 

Our hypothesis for Case 1
guarantees that $\psi_{\mu_{f(\xi_0)}\boxtimes\mu_{g(\xi_0)}}(
[0,1/[\mu_{f(\xi_0)}\boxtimes\mu_{g(\xi_0)}]])=
[0,x_{\xi_0}]$ (bijective correspondence)
and the only obstacle to the {\em analytic} extension of
$\psi_{\mu_{f(\xi_0)}\boxtimes\mu_{g(\xi_0)}}$ to 
$1/[\mu_{f(\xi_0)}\boxtimes\mu_{g(\xi_0)}]$ is the zero
derivative of $c_0$ in $x_{\xi_0}$. If we replace in the 
moment-generating function version of the $S$-transform equation (given
above) the variable $z$ by $\psi_{\mu_{f(\xi_0)}\boxtimes\mu_{g(\xi_0)}}(z)$ we obtain
$$
z\frac{\psi_{\mu_{f(\xi_0)}\boxtimes\mu_{g(\xi_0)}}(z)}{\psi_{\mu_{f(\xi_0)}\boxtimes\mu_{g(\xi_0)}}(z)+1}=
a_0(\psi_{\mu_{f(\xi_0)}\boxtimes\mu_{g(\xi_0)}}(z))
b_0(\psi_{\mu_{f(\xi_0)}\boxtimes\mu_{g(\xi_0)}}(z)).
$$
Denote for convenience 
$\omega_1=a_0\circ\psi_{\mu_{f(\xi_0)}\boxtimes\mu_{g(\xi_0)}},
\omega_2=b_0\circ\psi_{\mu_{f(\xi_0)}\boxtimes\mu_{g(\xi_0)}}.$
It has been shown in \cite{Biane} that $\omega_j$ extend analytically
to $\mathbb C\setminus[0,+\infty)$, preserve $\mathbb C^+$ and
increase the argument of the variable ($\arg \omega_j(z)\ge
\arg z,$ $z\in\mathbb C^+)$,
and in \cite{aihp} that their
restriction to the upper half-plane extends continuously to 
$\mathbb R$. In particular, $a_0,b_0$ extend continuously to 
$[0,x_{\xi_0}]$. Moreover, since 
$\psi_{\mu_{f(\xi_0)}\boxtimes\mu_{g(\xi_0)}}$ is real on $[0,
1/[{\mu_{f(\xi_0)}\boxtimes\mu_{g(\xi_0)}}]],$ $\omega_1,\omega_2$ must 
also be real on this same interval (see also \cite{aihp}), and thus
$a_0,b_0$ are continuous real functions on $[0,x_{\xi_0}]$. 
A direct application of the Schwarz reflection principle guarantees 
that $a_0,b_0$ extend analytically to a neighborhood of $[0,x_{\xi_0})
$ in $\mathbb C$, as claimed. It should be noted in addition that, as 
$\omega_j$ preserve half-planes (see \cite{Biane}),  
both $\omega_j$ are analytic on 
$[0,1/[\mu_{f(\xi_0)}\boxtimes\mu_{g(\xi_0)}])$, and so 
$a_0'(r),b_0'(r)>0$ for all $r\in[0,x_{\xi_0})$.

The analytic extension of $c_0$ around $x_{\xi_0}$ together with 
the fact that $c'_0(x_{\xi_0})=0$ guarantees that there exists
an $n>0$ so that
$z\mapsto(\psi_{\mu_{f(\xi_0)}\boxtimes\mu_{g(\xi_0)}}(z)
-\psi_{\mu_{f(\xi_0)}\boxtimes\mu_{g(\xi_0)}}(
1/[\mu_{f(\xi_0)}\boxtimes\mu_{g(\xi_0)}]))^n$ is analytic in a 
neighborhood of $1/[\mu_{f(\xi_0)}\boxtimes\mu_{g(\xi_0)}]$.
We shall denote $\mathcal S$ the Riemann surface 
determined by the corresponding $n^{\rm th}$ root. We shall
argue that, with the above notations, $\omega_1$ and $\omega_2$
extend analytically to a piece of $\mathcal S$
which projects onto a neighborhood of 
$1/[\mu_{f(\xi_0)}\boxtimes\mu_{g(\xi_0)}]$ (of course, excluding
$1/[\mu_{f(\xi_0)}\boxtimes\mu_{g(\xi_0)}]$).
We shall do this at first under the additional assumption that 
$\psi_{\mu_{f(\xi_0)}}$ and $\psi_{\mu_{g(\xi_0)}}$
do not share any critical {\em values}. 
Indeed, let us follow $\psi_{\mu_{f(\xi_0)}\boxtimes\mu_{g(\xi_0)}}
=\psi_{\mu_{f(\xi_0)}}\circ\omega_1$ along an arbitrary
path $p$ in $\mathcal S$ starting in the upper half-plane
close enough to $1/[\mu_{f(\xi_0)}\boxtimes\mu_{g(\xi_0)}]$.
Since $\arg\psi_{\mu_{f(\xi_0)}}(w)\ge\arg w$ for $w$ in the 
upper half-plane, $\psi_{\mu_{f(\xi_0)}\boxtimes\mu_{g(\xi_0)}}(z)$
will stay in the upper half-plane as long as $\omega_1(z)$ does.
Thus, we can then write $\omega_1(z)=\psi_{\mu_{f(\xi_0)}}^{-1}
(\psi_{\mu_{f(\xi_0)}\boxtimes\mu_{g(\xi_0)}}(z))$ whenever
$\omega_1(z)$ is still in $\mathbb C^+$ as $z$ runs through $p$. 
The only obstacle to the analytic extension of $\omega_1$ through 
a $z_k$ is a zero derivative of $\psi_{\mu_{f(\xi_0)}}$
in the point $\omega_1(z_k)\in\mathbb C^+$. 
Then $\psi_{\mu_{f(\xi_0)}}'(\omega_1(z_k))\omega_1'(z_k)=
\psi_{\mu_{f(\xi_0)}\boxtimes\mu_{g(\xi_0)}}'(z_k)$.
As without loss of generality 
$\psi_{\mu_{f(\xi_0)}\boxtimes\mu_{g(\xi_0)}}'(z_k)\neq 0$,
it follows that $\omega'_1$ has infinite limit in $z_k.$
Since also $\psi_{\mu_{g(\xi_0)}}'(\omega_2(z_k))\omega_2'(z_k)=
\psi_{\mu_{f(\xi_0)}\boxtimes\mu_{g(\xi_0)}}'(z_k)$, it 
follows from the $S$-transform 
equation and analytic continuation that necessarily 
$\omega_2'(z_k)$ is infinite, and moreover, the zeros of 
$\psi_{\mu_{f(\xi_0)}}'$ and $\psi_{\mu_{g(\xi_0)}}'$ in 
$\omega_1(z_k)$ and $\omega_2(z_k)$,
respectively, must be of the same order.
But since $\psi_{\mu_{f(\xi_0)}\boxtimes\mu_{g(\xi_0)}}
=\psi_{\mu_{f(\xi_0)}}\circ\omega_1
=\psi_{\mu_{g(\xi_0)}}\circ\omega_2$, we conclude that
$\psi_{\mu_{f(\xi_0)}}$ and $\psi_{\mu_{g(\xi_0)}}$
share the critical value 
$\psi_{\mu_{f(\xi_0)}\boxtimes\mu_{g(\xi_0)}}(z_k),$
contradicting our hypothesis. (For the origins of
this idea, see \cite{V-fe}.)

Thus, under the additional hypothesis regarding critical values,
we have shown that $\omega_1$ and $\omega_2$ extend to a 
simply connected domain $D$ of $\mathcal S$ which they map onto
$V\setminus[\omega_1(1/[\mu_{f(\xi_0)}\boxtimes\mu_{g(\xi_0)}]),
+\infty)$ and 
$V'\setminus[\omega_2(1/[\mu_{f(\xi_0)}\boxtimes\mu_{g(\xi_0)}]),
+\infty)$, respectively. Here $V$ and $V'$ are complex neighborhoods
of $\omega_1(1/[\mu_{f(\xi_0)}\boxtimes\mu_{g(\xi_0)}])$ and
$\omega_2(1/[\mu_{f(\xi_0)}\boxtimes\mu_{g(\xi_0)}])$, respectively.
Moreover, these extensions still satisfy 
the relations $\psi_{\mu_{f(\xi_0)}\boxtimes\mu_{g(\xi_0)}}
=\psi_{\mu_{f(\xi_0)}}\circ\omega_1
=\psi_{\mu_{g(\xi_0)}}\circ\omega_2$.
Since $\psi_{\mu_{f(\xi_0)}\boxtimes\mu_{g(\xi_0)}}$ extends
analytically to all of $\overline{D}$, we
use the fact that the moment generating functions
increase arguments to conclude that $\psi_{\mu_{f(\xi_0)}}$
and $\psi_{\mu_{g(\xi_0)}}$ must extend analytically to some interval
$(a_0(x_{\xi_0}),a_0(x_{\xi_0})+\varepsilon)$ and
$(b_0(x_{\xi_0}),b_0(x_{\xi_0})+\varepsilon)$, respectively,
and thus $a_0$ and $b_0$ must themselves extend 
analytically (and bijectively!) to some complex neighborhood
of $[0,x_{\xi_0}]$.

We have proved our claim under 
the additional assumption that 
$\psi_{\mu_{f(\xi_0)}}$ and $\psi_{\mu_{g(\xi_0)}}$
do not share any critical { values}. To complete the proof
of our claim that $a_0,b_0$ extend to some 
complex neighborhood of $[0,x_{\xi_0}]$ regardless of this 
condition being fulfilled, we only need to observe that
a translation of a measure $\mu$ by a real number $k$  
from the point of view of $w\mapsto\psi_\mu(w)=w^{-1}
G_\mu(w^{-1})-1$ into 
$w\mapsto w^{-1}G_\mu(w^{-1}-k)-1$. 
Then $[w^{-1}G_\mu(w^{-1}-k)-1]'=-w^{-2}G_\mu(w^{-1}-k)-
w^{-3}G_\mu'(w^{-1}-k)$. Thus, critical values change 
continuously in $k$. Since there can be at most countable 
critical values, we conclude that there exists a sequence
$k_m$ tending to zero so that the moment generating
functions of the translates of ${\mu_{f(\xi_0)}}$ by
$k_m$ and $\psi_{\mu_{g(\xi_0)}}$ have no common critical
values. Passing to the limit as $k_m\to0$ provides the required
answer.

But now the result under the assumption of Case 1 is proved;
by part (5) of our Proposition there exists a neighborhood $V$
of $[0,x_{\xi_0}]$ on which $a_n$ and $b_n$ extend analytically,
and by part (1) they converge to $a_0$ and $b_0$, respectively. 
By the $S$-transform property, $c_n\to c_0$ on $V$ as $n\to\infty$,
so there are points $d_n\in(0,+\infty)$ so that 
$c_n'(d_n)=0$ and $\lim_{n\to\infty}d_n=xญญญญญญญ_{\xi_0}$.
So $c_n(x_{\xi_n})=\lim_{k\to\infty}1/[\mu_{f(\xi_{n_k})}\boxtimes
\mu_{g(\xi_{n_k})}]=1/[\mu_{f(\xi_0)}\boxtimes\mu_{g(\xi_0)}]
=c_0(x_{\xi_0})$, as claimed.

\noindent{\bf Case 2:} For any $x>0$ in the domain of $c_0$,
we have $c_0'(x)>0$. If we denote as before $M_{\xi_0}$ to be the
upper bound of the domain of $c_0$, then 
$c_0(M_{\xi_0}):=\lim_{x\uparrow M_{\xi_0}}c_0(x)$ exists,
belongs to $(0,+\infty)$ (although $M_{\xi_0}$ might be equal to
$+\infty$) and equals $1/[\mu_{f(\xi_0)}\boxtimes\mu_{g(\xi_0)}]$.
By the $S$-transform equation, it follows that at least one
of $a_0$, $b_0$ must have $M_{\xi_0}$ as upper bound of the
domain of analyticity. Without loss of generality, assume that
$M_{\xi_0}=M_{f(\xi_0)}$ is the upper bound for the domain of 
analyticity of $a_0$. Condition (3) implies that 
$M_{f(\xi_n)}\to M_{f(\xi_0)}$ as $n\to\infty$. 
As the same condition holds for $g$, we easily
conclude that $\lim_{n\to\infty}\min\{M_{f(\xi_n)},M_{g(\xi_n)}\}
=M_{\xi_0}$ (limits considered in $[0,+\infty]$). 
If there is an $n_0\in
\mathbb N$ so that $c_n$ has no critical point in $(0,
\min\{M_{f(\xi_n)},M_{g(\xi_n)}\})$ for any $n\ge n_0$, then
condition (4) and the $S$-transform property allow us to conclude. 
Assume that for infinitely many $n$ the function $c_n$ has a 
critical point in $(0,\min\{M_{f(\xi_n)},M_{g(\xi_n)}\})$; call
the smallest of them $\zeta_n$. Then we know that $c_n(\zeta_n)=
1/[\mu_{f(\xi_n)}\boxtimes\mu_{g(\xi_n)}]$. Since $a_0,b_0$ are 
analytic on some complex neighborhood of $[0,M_{\xi_0})$, condition 
(5) tells us that for any $s\in[0,M_{\xi_0})$ there exists a 
neighborhood $V_s$ of $[0,s]$ in $\mathbb C$ so that, from a certain 
$n$ on, all $a_n,b_n$ have an analytic extension to 
$V_s$. If there is a subsequence $\{\zeta_{n_k}\}_k$ which converges 
to a point $r<M_{\xi_0}$, then $c_{n_k}$ converges to $c_0$
uniformly on compacts of $V_r$ by condition (1) and thus 
$r$ is a critical point of $c_0$, contradicting the assumption
of Case 2. 
The case when $\zeta_n$ converges to $M_{\xi_0}$ as $n\to\infty$
is covered by condition (4): indeed, this condition implies that
$c_n(\zeta_n)=\frac{1+\zeta_n}{\zeta_n}a_n(\zeta_n)b_n(\zeta_n)\to
\frac{1+M_{\xi_0}}{M_{\xi_0}}a_0(M_{\xi_0})b_0(M_{\xi_0})=c_0(M_{\xi_0})$ as $n\to\infty$. (Here we use the obvious convention
$\frac{1+\infty}{\infty}=1.)$ This concludes our proof.
\end{proof}
We would like to emphasize that some of the conditions of the above
proposition can be weakened or replaced with conditions of a different
nature: we use this set of
conditions simply because it covers a conveniently large 
family of distributions for our purposes.

\begin{corollary}\label{cor:boxplus-cont-util}
If $\mu$ is a fixed compactly supported probability measure on $[0,+
\infty)$, $a=(a_1,\dots,a_m)\in[0,+\infty)^m\setminus\{(0,\dots,0)\}$, 
$t=(t_1,\dots,t_{m})\in(0,1)^{m}\cap\Delta_m$ 
(so $(t_1,\dots,t_m)$ satisfy $\sum_{j=1}^mt_j=1$), and 
$\nu(a,t)=\sum_{j=1}^{m}t_j\delta_{a_j}$, then the correspondence
$([0,+\infty)^m\setminus\{(0,\dots,0)\})\times\Delta_m\ni(a,t)\mapsto[\mu\boxtimes\nu(a,t)]$ is continuous.
\end{corollary}
\begin{proof}
We shall apply the previous proposition, with the identifications
$X=([0,+\infty)^m\setminus\{(0,\dots,0)\})\times\Delta_m$,
$f$ the constant function taking value $\mu$, and 
$g(a,t)=\nu(a,t)=\sum_{j=1}^{m}t_j\delta_{a_j}$. One checks that $f$  satisfies
all conditions from the proposition above.The weak continuity of $g$ is
equally clear, as is the continuity of the correspondence
$(a,t)\mapsto[\nu(a,t)].$ Observing that $\psi_{\nu(a,t)}$
maps $(-\infty,1/\max\{a_1,\dots,a_m\})$ monotonically 
and bijectively into $\left(\left[\sum_{j\colon a_j=0}t_j\right]-1,
+\infty\right)$ assures us that the upper bound of the domain of 
$\psi_{\nu(a,t)}^{-1}$ is constantly equal to infinity, and hence
continuous, and moreover, $\psi_{\nu(a,t)}^{-1}$ maps plus infinity
into $1/[\nu(a,t)]$, guaranteeing the continuity of 
$\psi_{\nu(a,t)}^{-1}(M_{g(a,t)})$, and hence the
verification of conditions (3) and (4). Condition (5) is 
verified by the constant function $f$. For $g$ one only needs to 
recall the observations following equation \eqref{S}
to note that indeed, given any compact subset of $X$,
there is a complex neighbourhood of $[0,+\infty)$ on which $\psi_{\nu(a,t)}^{-1}$ is analytic for all $(a,t)$ in the given compact set.
Thus, a stronger version of condition (5) is satisfied by $g$.
Applying the above proposition allows us to conclude.
\end{proof}

\section{Almost sure convergence of norms of random matrices}
\label{sec:rmt}

Let GUE be the Gaussian Unitary Ensemble, i.e. the probability measure on $\M_{n}(\C)$ with support on
self-adjoint matrices and density proportional to $\exp (-n/2 \trace (A^{2}))\d{A}$.
The following theorem was obtained in the seminal paper \cite{HT} by Haagerup and Thorbj\o rnsen:

\begin{theorem}\label{thm:HT}
Let $X_n,Y_n$ be two i.i.d GUE random variables on $\M_{n}(\C)$ and $P$ be a non-commutative polynomial in two variables.
Then, almost surely as $n \to \iy$,
$$\norm{P(X_{n},Y_{n})}_\iy \to \norm{P(x,y)}$$
where $x,y$ are free semi-circular elements in a finite von Neumann algebra.
\end{theorem}

The aim of this section is to build on Theorem \ref{thm:HT}, and extend it to some specific non-commutative monomials of random matrices with prescribed spectra. 

We recall that if $X$ is an $n$-dimensional self-adjoint matrix, its \emph{eigenvalue counting measure} is $n^{-1}\sum_{i=1}^{n}\delta_{\lambda_{i}}$ where
$\lambda_{i}$ are the eigenvalues of $X$.
For any probability measure $\mu$ on the real line, its \emph{distribution function} is defined as $f_{\mu}: t  \mapsto \mu ((-\infty,t])$.

For the purposes of this section, we will say that a sequence of distribution functions $f_{n}$ tends to a distribution function $f$ iff for all $\e >0$,
 there exists an $n_{0}$ such that for all $n\geq n_{0}$,
\begin{equation}
	\forall t\in\R,\quad f(t-\e )-\e \leq f_{n}(t)\leq f(t+\e ) +\e
\end{equation}

\begin{theorem}\label{thm:RMT-norm}
Let $A_{n},B_{n}$ be independent positive self-adjoint random matrices in $\M_{n}(\C)$, such that at least one of $A_{n}$ or $B_{n}$ has a distribution invariant under unitary conjugation. 
Let $f_{n}$ be the distribution function of $A_{n}$
and $g_{n}$ be the distribution function of $B_{n}$. Assume that the (a priori random) distribution  functions $f_{n},g_{n}$ converge almost surely respectively to $f,g$ which are distribution functions of two self-adjoint, bounded and freely independent random variables $x$ and $y$. 
Assume also that the operator norm of $A_n$ (resp. $B_n$) converges to the operator norm of $x$ (resp. $y$).

Then, almost surely as $n \to \infty$,
$$||A_{n}B_{n}|| \to ||xy||.$$
\end{theorem}

Similar results have been obtained recently by C. Male \cite{male}. However, our results do not clearly follow from his. 
We also believe that the above theorem could be proved directly with determinantal processes methods, see e.g. \cite{defosseux, metcalfe},
at least in the case where one of the operators is a projection.

Note also that 6 months after the first version of this paper was completed, one author and C. Male used one key ingredient introduced in the proof
below to prove a substantial extension of  Theorem \ref{thm:HT} in the unitary case, see \cite{collins-male}.
The more recent main result of \cite{collins-male}, even though quite general, does not imply directly Theorem \ref{thm:RMT-norm} because our assumptions on
the spectrum of $A_{n},B_{n}$ are not as restrictive as in \cite{collins-male}.

\begin{proof}[Proof of Theorem \ref{thm:RMT-norm}]

Without loss of generality, we can assume that both $A_{n}$ and $B_{n}$ have distributions which are invariant under unitary conjugation (indeed, replacing the pair $(A_{n},B_{n})$ by $(V_{n}A_{n}V_{n}^{*},V_{n}B_{n}V_{n}^{*})$ where
$V_{n}$ is a Haar distributed random state independent from $(A_{n},B_{n})$ does not change the hypotheses nor $||A_{n}B_{n}||$, but enforces unitary invariance on both $A_{n}$ and $B_{n}$.

The main idea is to adapt Theorem \ref{thm:HT} to our case by showing that
in the case where $P$ of Theorem \ref{thm:HT} is of the form $P_{1}(x)P_{2}(y)$, it extends to the situation where $P_{1},P_{2}$ are any nondecreasing bounded functions.

In this proof we consider a pair $X_{n}, Y_{n} \in \M_n(\C)$ of i.i.d GUE random matrices and we split our proof into three steps. In the first two steps, we show how we can replace in Theorem \ref{thm:HT} polynomials by 
real, non-decreasing, c\`adl\`ag, non-negative and bounded functions.
In the third step, we show how, via functional calculus, we can modify the pair  $(X_{n},Y_{n})$ into a pair that has the same distribution as $(A_{n},B_{n})$.

\noindent \textbf{Step I.} First, we prove that if $P$ is any real positive polynomial and $S_0$ is a distribution function (real, non-decreasing, c\`adl\`ag and positive), then, for all $\e >0$, for a fixed small enough neighborhood $\mathcal V$ of $S_0$, almost surely, there exists $n_0 \in \N$ such that, for all $n \geq n_0$ and for all $S \in \mathcal V$,
\begin{equation}\label{eq:Step-I}
	| \norm{P(X_n)S(Y_n)P(X_n)}_\iy -  \norm{P(x)S(y)P(x)} | < \e,
\end{equation}
were $x,y$ are free semicircular elements in a $\mathrm{II}_1$ factor.

For $\e>0$, we introduce the functions $S^+_0(x) = S_0(x+\e)+\e$ and $S^+_0(x) = S_0(x-\e) - \e$. Clearly, we have $S^-_0 < S_0 < S^+_0$. Moreover, since the neighborhood $\mathcal V$ of $S_0$ can be chosen as small as we need to, we can choose it in such a way that for all $S \in \mathcal V$, the jumping points of $S$ are at distance at most $\e / 100$ from the jumping points of $S_0$. By Stone-Weierstrass theorem, there exist polynomials $Q^\pm$ such that, on the interval $[-3,3]$, for all $S \in \mathcal V$, $S^-_0 < Q^- < S < Q^+ < S^+_0$ (see Figure \ref{fig:Step-I}).
\begin{figure}[htbp]
\centering
\includegraphics{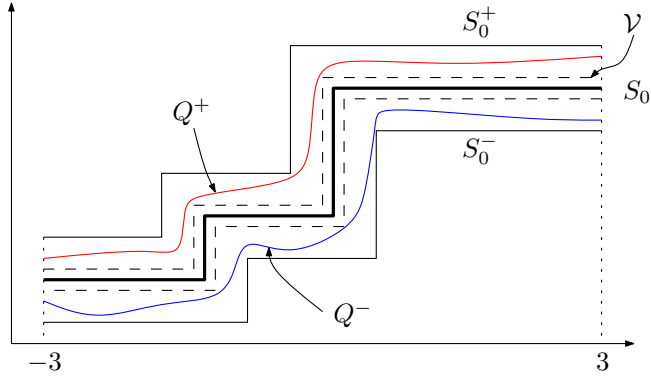}
\caption{Bounding distribution functions uniformly by polynomials.}
\label{fig:Step-I}
\end{figure}

The fact that almost surely the eigenvalues of $X_n$ and $Y_n$ are included in $[-3,3]$ as $n\to\infty$ implies that almost surely, for all $S \in \mathcal V$ and for $n$ large enough, $S(Y_n)< Q^+(Y_n)$. 
Therefore, almost surely for $n$ large enough, $P(X_n)S(Y_n)P(X_n)<P(X_n)Q^+(Y_n)P(X_n)$ and thus, using positivity,
$$\norm{P(X_n)S(Y_n)P(X_n)}_\iy<\norm{P(X_n)Q^+(Y_n)P(X_n)}_\iy.$$
However, $P$ and $Q^+$ are polynomials, therefore we can use Theorem \ref{thm:HT} to claim that
$\norm{P(X_n)Q^+(Y_n)P(X_n)}_\iy \to \norm{P(x)Q^+(y)P(x)}$.
We have shown that, almost surely for $n$ large enough,
$$\limsup_{n \to \iy} \norm{P(X_n)S(Y_n)P(X_n)}_\iy \leq \norm{P(x)Q^+(y)P(x)}.$$
In the von Neumann algebra generated by two free semicircular elements 
$x,y$, we have the inequality $P(x)Q^+(y)P(x)\leq P(x)
S^+_0(y)P(x)$, therefore 
$$\limsup_{n \to \iy} \norm{P(X_n)S(Y_n)P(X_n)}_\iy \leq\norm{P(x)
S^+_0(y)P(x)}.$$
Since this is true for all $\e >0$ and the norm is continuous 
according to Corollary \ref{cor:boxplus-cont-util}, by letting $\e \to 0$ we get
$$\limsup_{n \to \iy} \norm{P(X_n)S(Y_n)P(X_n)}_\iy \leq \norm{P(x)S(y)P(x)}.$$

A similar argument, using this time $Q^-$ and $S_0^-$ to bound from below elements $S \in \mathcal V$ proves the other inequality and completes the first step of the proof. Note however that the lower bound could have been obtained without using Theorem \ref{thm:HT}; indeed, one can use Voiculescu's result for the convergence of empirical spectral distributions of random matrices to conclude.

\noindent \textbf{Step II.} The second part of our proof is to show that that one can replace the polynomial $P$ in equation \eqref{eq:Step-I} by another function $T$ chosen from a neighborhood $\mathcal W$ of a given distribution function $T_0$. First, note that in Step I, one can interchange the roles of the polynomial $P$ and the step function $S$ by using the $C^*$ algebra equality, $\norm{a}^2=\norm{aa^*}$. Hence, $\norm{S(X_n)P(Y_n)S(X_n)}_\iy$ converges to $\norm{S(x)P(y)S(x)}$. Then, we employ the same technique as in Step I: we bound any element $T \in \mathcal W$ by fixed polynomials $P^\pm$ and we use Step I to conclude. Note that in the first two steps of the proof we have considered GUE matrices $X_n$ and $Y_n$.

\noindent \textbf{Step III.} 
In this final step, we consider our original sequence $(A_{n},B_{n})$ and show that our conclusion holds for it. For the purpose of its study, we introduce an auxiliary pair $(X_{n},Y_{n})$ of two 
 i.i.d Gaussian ensembles.
It is known that with probability one, all its eigenvalues have multiplicity one. So without loss of generality, we will assume that our instance
of $(X_{n},Y_{n})$ does not have multiplicity in its eigenvalues. Similarly, we assume that the normalized eigenvalue counting function of $X_n$ and $Y_n$ converges towards the semi-circle
and that their operator norm converges to 2. 
It is also possible to do so without loss of generality because of the well known convergence properties of the Gaussian unitary ensembles \cite{HT}.

From this, it follows that there exists two non-decreasing c\`adl\`ag functions $f_n$, $g_n$ such that the eigenvalues of $f_n(X_n)$ are the same as those of $A_n$ and
the eigenvalues of $g_n(Y_n)$ are the same as those of $B_n$. 

The functions $f_n$ and $g_n$ are not unique and are random, but it follows from our hypotheses on the limiting distributions of $A_n,B_n$ and 
our choice of $X_n,Y_n$ that it is possible to make sure that $f_n$ and $g_n$ converge uniformly.

Let us denote by $a_1\geq \ldots \geq a_n$ the eigenvalues of $A_n$, $b_1\geq \ldots \geq b_n$ the eigenvalues of $B_n$, 
$x_1> \ldots > x_n$ the eigenvalues of $X_n$, and $y_1> \ldots > y_n$ the eigenvalues of $Y_n$ (note that we
make a small abuse of notation for the sake of simplicity, and omit in the notation the dependence in $n$).
It follows from the above that for all $i$, $f_n(x_i)=a_i$ and $g_n(y_i)=b_i$.

Next, let us introduce the decomposition 
$A_n=U_a \mathrm{diag} (a_1,\ldots ,a_n)U_a^*$ and similarly for $B_n,X_n,Y_n$.
It is known that it is possible to make a choice for $U_a$ (resp.  $U_b,U_x,U_y$)  that depends from $A_n$ (resp. $B_n,X_n,Y_n$) in a measurable way.

Let $\tilde X_n=U_a \mathrm{diag} (x_1,\ldots ,x_n)U_a^*$ and $\tilde Y_n=U_b \mathrm{diag} (y_1,\ldots ,y_n)U_b^*$.

The matrices 
$\tilde X_n,\tilde Y_n$ are  random matrices and they have the property that $f_n(\tilde X_n)=A_n$ and $g_n(\tilde Y_n)=B_n$. 
Besides, they are independent from each other. Finally, they both follow the GUE distribution because the latter is known to be determined
by three criteria that are obviously satisfied in the construction of $\tilde X_n,\tilde Y_n$, namely:
(a) the distribution of its eigenvalues is the correct one, (b)  its eigenvalues and its eigenvectors are independent, and (c) its eigenvectors are distributed according to the invariant measure. 

We conclude the proof by an application of Step II to the matrices 
$\tilde X_n,\tilde Y_n$ with the functions $f_n,g_n$.

\end{proof}

\section{Asymptotic behaviour of $K_{n,k,t}$}\label{sec:main}

We now introduce the convex body $K_{k,t}\subset \Delta_k$ as follows:
\begin{equation}\label{eq:convex}
	K_{k,t}:=\{ \lambda \in\Delta_{k} \; |\; \forall a\in\Delta_{k} , \scalar{\lambda}{a} \leq \normt a
	\},
\end{equation}
where $\scalar{\cdot}{\cdot}$ denotes the canonical scalar product in $\R^k$. 
We shall show in Theorem \ref{convex-norm} that this set is intimately related to the $(t)$-norm: $K_{k,t}$ is the intersection of the dual ball of the $(t)$-norm with the probability simplex $\Delta_k$. Since it is defined by duality, $K_{k,t}$ is the intersection of the probability simplex with the half-spaces
$$
H^+(a, t) = \{x \in \R^k \; | \; \scalar{x}{a} \leq \normt{a}\}
$$
for all directions $a \in \Delta_k$. Moreover, we shall show in Theorem \ref{thm:lln-largest-eigenvector} that every hyperplane $H(a, t) = \{x \in \R^k \; | \; \scalar{x}{a} = \normt{a}\}$ is a supporting hyperplane for $K_{k,t}$.

\subsection{A set of probability one and statement of the results}

Let $(\Omega,\mathcal F,\P)$ be a probability space in which the sequence or random vector subspaces $(V_{n})_{n\geq 1}$ is defined.
Since we assume that the elements of this sequence are independent, we may assume that
$\Omega = \prod_{n\geq 1}\Gr_N(\C^k\otimes \C^n)$
and $\P=\otimes_{n\geq 1}\mu_{n}$ where $\mu_{n}$ is the invariant measure on the Grassman manifold $\Gr_N(\C^k\otimes \C^n)$.
Let $P_{n} \in \M_{nk}(\C)$ be the random orthogonal projection whose image is $V_{n}$. For two positive sequences $(a_n)_n$ and $(b_n)_n$, we write $a_n \ll b_n$ iff $a_n / b_n \to 0$ as $n \to \iy$.
\begin{proposition}\label{prop:set-of-proba-one}
Let $\nu_{n}$ be a sequence of integers satisfying $\nu_{n}\ll n$.
Almost surely, the following holds true: 
for any self-adjoint matrix $A\in \M_{k}(\C)$, 
the $\nu_{n}$-th largest eigenvalues of $P_{n}(A\otimes I_{n})P_{n}$
converges to $||a||_{(t)}$
where $a$ is the eigenvalue vector of $A$. 
This convergence is uniform on any compact set of $\M_{k}(\C)_{sa}$.
\end{proposition}

\begin{proof}
For any self-adjoint $A\in \M_{k}(\C)$, the almost sure convergence follows from Theorem 
\ref{thm:RMT-norm}
and from Theorem \ref{libre}.

Let $A_{l}$ be a countable family of self-adjoint matrices in $\M_{k}(\C)$ and assume that
their union is dense in the operator norm unit ball.
By sigma-additivity, 
the property to be proved holds almost-surely simultaneously for
all  $A_{l}$'s. 

This implies that the property holds for all $A$ almost-surely, as the $j$-th largest eigenvalue
of a random matrix is a 
Lipschitz function for the operator norm on the space of matrices. 
\end{proof}

The set on which the conclusion of the above proposition holds true will be denoted by $\Omega'$ and we therefore have $\P(\Omega ')=1$. Technically, $\Omega'$ depends on $\nu_{n}$ but in the proofs, we won't need to keep track of this dependence as  $\nu_n$ will be a fixed sequence.

The main result of our paper is the following characterization of the asymptotic behavior of the random set $K_{n,k,t}$. We show that this set converges, in a very strong sense, to the convex body $K_{k,t}$.

\begin{theorem}\label{thm:main}
\label{limite}
Almost surely, the following holds true:
\begin{itemize}
\item Let $\mathcal{O}$ be an open set in $\Delta_{k}$ containing $K_{k,t}$.
Then,  for $n$ large enough, $K_{n,k,t}\subset \mathcal{O}$.
\item Let $\mathcal K$ be a compact set in the interior of $K_{k,t}$.
Then, for $n$ large enough, $\mathcal K \subset K_{n,k,t}$.
\end{itemize}
\end{theorem}

The proof of this theorem goes according to the following non-standard scheme: the first inclusion follows a strategy developed in
\cite{collins-nechita-2} and improves on it.
This is the object of Theorem \ref{thm:upper-bound}. Revisiting the strategy of proof of Theorem \ref{thm:upper-bound} gives rise to a 
result about eigenvectors of random matrices, as stated in Theorem \ref{thm:lln-largest-eigenvector} below, and in turn, 
Theorem \ref{thm:lln-largest-eigenvector} is needed to prove the second part of Theorem
\ref{thm:main}. This is the purpose of Theorem \ref{thm:lower-bound}.

Note that all the statements above are of almost sure nature. 
At first sight this looks unnatural because there is no assumption on the probability space on which the family of random matrices indexed by 
the dimension is defined. The only assumptions are on the $n$-dimensional marginals. The fact that the results hold with probably
one on any probability probability space having the appropriate marginals follows from arguments of Borel-Cantelli type. 

Instead of stating a result of convergence almost surely, it is also possible, in the spirit of e.g. \cite[Theorem 2.1.1]{agz}, 
to write down a theorem of convergence in probability. The benefit of doing so is that one does not need to bother to
realize all random matrices in a same probability space. Such a result actually follows from the above Theorem.
We could have chosen such an approach, but we felt that the technical details of the proof would have been more involved (in our proof
we intersect countably many probability one measurable subsets of an appropriate probability space).
Note also that Anderson, Guionnet and Zeitouni also state results of almost sure convergence  (see  for example \cite{agz} Exercise 2.1.16).
Similarly, in the original results by Haagerup and Thorbj{\o}rnsen, the convergence results are of almost sure nature.

A byproduct of the first part of  the above theorem, and a necessary step towards its second part is the following result, of independent interest in random matrix theory:

\begin{theorem}
\label{thm:lln-largest-eigenvector}
Consider a matrix $A = \diag(a)$ whose eigenvalue vector is $a \in \R^k$ and let $\nu_{n}$ be a sequence of integers satisfying $\nu_{n}\ll n$.
We assume that all eigenvalues of $A$ are simple.

Let $x^{(n)}$ be the unital eigenvector corresponding to the $\nu_{n}$-th largest eigenvalue of $P_{n}(A\otimes I_{n})P_{n}$, which admits a 
singular value decomposition
$$x^{(n)}=\sum_{i=1}^{k}\sqrt{\lambda_{i}^{(n)}}e_{i}^{(n)}\otimes f_{i}^{(n)}.$$
Then, almost surely, for each $i=1, 2, \ldots, k$, $e_{i}^{(n)}$ converges to the eigenvector corresponding to the $i$-th largest eigenvalue of $A$
(modulo a phase change).
Moreover, if $\lambda$ is the 
exposed point of $K_{k,t}$ such that the supporting hyperplane is defined by the direction $a$, then, almost surely
$$\lim_{n \to \iy} \lambda^{(n)} = \lambda.$$ 
\end{theorem}

This theorem has its own interest from the random matrix point of view. 
Indeed it can be seen as a law of large numbers for the $\U(k)/\U(1)^{k}$ and the $\R^{k}$ components of the 
singular value decomposition of the eigenvectors. Even though many laws of large numbers have been obtained for eigenvalues, 
not much is known about the structure of eigenvectors (except \cite{nader}, \cite{benaych-rao} and references therein).

\subsection{Upper bound}
\label{sec-upperbound}
The first part of Theorem \ref{limite} is the following result:

\begin{theorem}
\label{thm:upper-bound}
Let $\mathcal{O}$ be an open set in $\Delta_{k}$ containing $K_{k,t}$.
Then almost surely, for $n$ large enough, $K_{n, k, t}\subset \mathcal{O}$.
\end{theorem}
This result provides almost surely an upper bound for the set $K_{n,k,t}$.
The proof of this theorem  relies on Theorem \ref{thm:RMT-norm}, and on two lemmas, that are adapted from \cite{collins-nechita-2} and which we state and prove below.

\begin{lemma}
\label{lem:elementaire}
Let $Q \in \M_n(\C)$ be a self-adjoint projection and $R \in \M_n(\C)$ be a self-adjoint element. Then
\begin{equation}
	\norm{QRQ}_{\iy}=\max_{x \in \Image Q} \trace (P_xR),
\end{equation}
where $P_x$ denotes the orthogonal projection on the one-dimensional space $\C x$.
\end{lemma}

For two matrices $A, B \in \M_k(\C)$, we write $A \sim B$ if there exists a unitary operator $U \in \U(k)$ such that $A = U B U^*$. 
For a vector $x \in \C^k \otimes \C^n$ with Schmidt coefficients 
 $\lambda_1 \geq \lambda_2 \geq \cdots \geq  \lambda_k\geq 0$, 
 and an element $a \in \Delta_{k}^{\downarrow}$, we introduce the notation
 $$s^a (x) = a_1\lambda_1 +\ldots + a_k\lambda_k = \scalar{a}{\lambda}. $$
 Similarly, for a matrix  $A\in \M_{k}(\C)$,
we introduce the notation
$$s^{A}(x) :=\trace (P_x \cdot  A \otimes \I_n) = \trace(\trace_n P_x \cdot A),$$
where $\trace_{n}=\id_{k}\otimes\trace$ is the non-normalized conditional expectation
$\M_{nk}(\C)\to\M_{k}(\C)$.

\begin{lemma}
\label{lem:sup-B}
Let $A$ be a self-adjoint matrix with ordered eigenvalue vector $a \in\Delta_{k}^{\downarrow}$.
For each $x \in \C^k \otimes \C^n$,
the following holds true:
$$s^a (x)= \max_{A' \sim A} s^{A'}(x).$$
\end{lemma}
\begin{proof}

For two matrices $A,B\in\M_{k}(\C)$ with respective eigenvalues $\mu_{1}\geq\ldots \geq\mu_{k}\geq 0$ and
$\lambda_{1}\geq\ldots \geq\lambda_{k}\geq 0$, it follows from the min-max theorem that
 $$\sum_{i}\lambda_{i}\mu_{i}=\max_{A'\sim A}\trace (A'B).$$

Letting $B=\trace_n P_x$, the above observation implies:
\begin{equation}
	s^a (x)=\max_{U\in \U(k)}\trace ( UAU^*\trace_n P_x ) = \max_{A' \sim A}\trace ( A'\trace_n P_x ).
\end{equation}
The conditional expectation property of the partial trace implies that 
\begin{equation}
	s^a (x)= \max_{A' \sim A} \trace(P_x \cdot A' \otimes \I_n) = \max_{A' \sim A} s^{A'}(x).
\end{equation}
\end{proof}
Since $k$ is a fixed parameter of our model, in order to compute the maximum in Lemma \ref{lem:elementaire} over the unitary orbit indexed 
by $\U(k)$, we can pick a finite but large enough number of elements of the corresponding orbit to obtain a good approximation of the maximum:

\begin{lemma}
\label{lem:estimate-grassmannian}
For a fixed self-adjoint matrix $A \in \M_k(\C)$ with eigenvalue vector $a \in \R^k$ and for all $\e >0$, there exist a finite number of matrices $B_1,\ldots ,B_l$ self-adjoint and
conjugated to $A$, such that, for all $x \in \C^{nk}$,
\begin{equation}
	\max_{i=1}^l \trace(P_x  \cdot B_i \otimes \I_n) \leq s^a(x) \leq \max_{i=1}^l \trace(P_x \cdot B_i \otimes \I_n) + \e.
\end{equation}
\end{lemma}

\begin{proof}
We only need to prove the second inequality, the first one being a direct consequence of Lemma \ref{lem:sup-B}. 
Since the orbit under unitary conjugation of a self-adjoint matrix $A$ is compact for the metric $d(B, B') = \norm{B- B'}_{\iy}$, 
for all $\e >0$ there exists a covering of the orbit by a finite number of balls of radius $\e$ centered in $B_1, B_2, \ldots, B_l$. 
Fix some $x \in \C^{nk}$ and consider the element $B$ in the orbit of $A$ for which the maximum in the definition of $s^a(x)$ 
is attained. The matrix $B$ is inside some ball centered at $B_i$ and we have
\begin{equation}
	\begin{split}
		\trace(P_x\cdot B \otimes \I_n) &\leq  \trace(P_x \cdot B_i \otimes \I_n) + \module{\trace\left[ P_x\cdot (B_i - B)\otimes \I_n \right] } \\
		&\leq \trace(P_x \cdot B_i \otimes \I_n) + \norm{B_i-B}_{\iy}\leq \trace(P_x \cdot B_i \otimes \I_n) + \e
	\end{split}
\end{equation}
and the conclusion follows.
\end{proof}

Now we are ready to prove Theorem \ref{thm:upper-bound}.

\begin{proof}[Proof of Theorem \ref{thm:upper-bound}]
For a given open neighborhood $\mathcal O$ of $K_{k,t}$, one can find a small positive constant $\e$ and a finite 
number of ordered probability vectors $a_1, a_2, \ldots, a_L \in \Delta_k^\downarrow$ such that 
\begin{equation}
	K_{k,t} \subset \bigcap_{i = 1}^L \left\{z \in \Delta_k \; |\; \scalar{z^\downarrow}{a_i} \leq \normt{a_i}\right\} \subset \bigcap_{i = 1}^L \left\{z \in \Delta_k \; |\; \scalar{z^\downarrow}{a_i} \leq \normt{a_i}+\e\right\} \subset \mathcal O.
\end{equation}
Note that only the last inclusion is non-trivial in the above equation.
Consider a positive self-adjoint matrix $A \in \M_k(\C)$ with eigenvalue vector $a \in \Delta_k^\downarrow$ and $V_{n}$ a random vector space of dimension $N\sim tnk$. According to Theorem \ref{thm:RMT-norm}, almost surely, we have that
\begin{equation}
	\lim_{n \to \iy} \norm{P_{V_{n}} \cdot (A \otimes \I_n) \cdot P_{V_{n}}}_\iy= \normt{a}. 
\end{equation}
By Lemma \ref{lem:sup-B}, for every such subspace $V$, one also has that
\begin{equation}
	\max_{\substack{x \in V \\ \norm{x} = 1}} s^a(x) = \max_{\substack{x \in V \\ \norm{x} = 1}} \max_{B \sim A} 
	\trace(P_x \cdot B \otimes \I_n).
\end{equation}

Using the compactness argument in Lemma \ref{lem:estimate-grassmannian}, one can consider (at a cost of $\e$) only a finite number of matrices $B$:
\begin{equation}
	\max_{\substack{x \in V \\ \norm{x} = 1}} s^a(x) \leq \max_{i=1}^l \max_{\substack{x \in V \\ \norm{x} = 1}} \trace(P_x\cdot (B_i \otimes \I_n)) + \e = \max_{i=1}^l \norm{P_{V_{n}} B_i \otimes \I_n P_{V_{n}}}_\iy + \e.
\end{equation}

After after applying Theorem \ref{thm:RMT-norm} to each of the pairs $(B_j,P_{V_n})$, $1\le j\le l$,
one has that, almost surely,
\begin{equation}
	\limsup_{n \to \iy} \max_{\substack{x \in V \\ \norm{x} = 1}} s^{a}(x)\leq \normt{a} +\e.
\end{equation}

Using $L$ times the previous line of reasoning, by letting $a=a_i$ for $i=1, \ldots, L$, we obtain that, almost surely, for $n$ large enough, 
\begin{equation}
	K_{n,k,t} \subset \bigcap_{i = 1}^L \left\{z \in \Delta_k \; |\; \scalar{z^\downarrow}{a_i} \leq \normt{a_i}+\e\right\} \subset \mathcal O.
\end{equation}
\end{proof}

\subsection{Lower bound}

We start with the proof of Theorem \ref{thm:lln-largest-eigenvector}, needed for the second part of our main result, Theorem \ref{thm:main}.

\begin{proof}[Proof of Theorem \ref{thm:lln-largest-eigenvector}]
Since the set $\Omega'$ introduced after Proposition \ref{prop:set-of-proba-one} has probability one, we may pick a sequence $(V_{n})_{n\in\N}$ in the set $\Omega'$ defined after the Proposition \ref{prop:set-of-proba-one}. 

Let us consider the eigenvector $x^{(n)}$ of the $\nu_n$-th largest eigenvalue of $P_n (A \otimes \I_n) P_n$ and write its singular value (or Schmidt) decomposition:
$$x^{(n)}=\sum_{j=1}^{k}\sqrt{\lambda_{j}^{(n)}} e_{j}^{(n)}\otimes f_{j}^{(n)}.$$

To start, notice that since the range of the matrix $P_n (A \otimes \I_n) P_n$ is a subspace of $V_n$, one must have $x^{(n)} \in V_n$. It has been shown in the proof of Theorem \ref{thm:upper-bound} that
for any open set $\mathcal{O}$ containing $K_{k,t}$, the probability vector $\lambda^{(n)}$ is in $\mathcal{O}$, for $n$ large enough.
 
Using the fact that $x^{(n)}$ is the eigenvector corresponding to $\mu_{n}$, the $\nu_{n}$-th largest eigenvalue of $P_{n}(A\otimes I_{n})P_{n}$, we obtain that
$$P_{n}(A\otimes I_{n})P_{n}P_{x^{(n)}}=\mu_{n}P_{x^{(n)}}.$$
Recall that (Proposition \ref{prop:set-of-proba-one}) $\mu_{n}\geq \normt{a} - \e$ for $n$ large enough, thus
$$\trace( P_{n}(A\otimes I_{n})P_{n}P_{x^{(n)}} )\geq \normt{a}\trace P_{x^{(n)}} - \e,$$
where $a\in \Delta_{k}$ is the eigenvalue of $A$. 
Since $x^{(n)}\in V_n = \Image P_{n}$, it follows that $P_{n}P_{x^{(n)}}=P_{x^{(n)}}$. In addition, using the fact that $\trace P_{x^{(n)}}=1$, one obtains the following lower bound:
$$s^{A}(x^{(n)})\geq \normt{a} - \e.$$

This implies that for $n$ large enough, 
$$\lambda^{(n)}\in \mathcal{O}\cap \{z \; | \;  \scalar{z^\downarrow}{a} \geq \normt{a} - \e\}.$$

Hence, the hyperplane $H_a = \{z  \; | \; \scalar{z^\downarrow}{a} \leq ||a||_{(t)}\}$ is a supporting hyperplane for the convex set $K_{k,t} \subset \Delta_k$.

If $z$ is an \emph{exposed} point of $K_{k,t}$, defined by a hyperplane $H_a$ which intersects $K_{k,t}$ only at $z$, then $\lambda^{(n)}$ converges to the exposed point $z$, showing the first part of the result.

Next, we study the convergence of the Schmidt vectors $e_{i}^{(n)} \in \R^k$.
Let $B \sim A$ be a self-adjoint matrix in $\M_{k}(\C)$ with same eigenvalues as $A$. It follows from the proof of Theorem \ref{thm:upper-bound} that $s^B(x^{(n)}) \leq \normt{a} + \e$ for large enough $n$.

Hence, the function 
$$B \mapsto s^B(x^{(n)}) = \trace(B \cdot \trace_n P_{x^{(n)}})$$
is $2\e$-close to its maximum  at $B = A$. Using the general fact that the real 
function $$\U(k)\ni U \mapsto \trace (AUBU^{*})$$ 
is continuous and has only one maximum, achieved when the eigenvectors of $UAU^*$ are 
parallel to the eigenvectors of $B$ (and respecting the order of the eigenvalues), we can conclude the proof of the lower bound.
\end{proof}

The next result is an improvement over Theorem \ref{thm:lln-largest-eigenvector} and shows that we do not need to restrict 
ourselves to a single eigenvector $x^{(n)}$ but that we can choose $x$ in a vector space of arbitrary size (prescribed in advance) 
such that the conclusions of the above theorem still hold for $x$. This fact will be useful in the final step of the proof of the 
Theorem \ref{thm:lower-bound}, as it allows to perform a Gram-Schmidt orthogonalization procedure.

\begin{proposition}
\label{prop:lln-subspace}
Let $\lambda$ be an exposed point of $K_{k,t}$ and let $a$ be a direction of the supporting hyperplane tangent at $\lambda$. 
Then, for any $\e >0$ and any integer $l$, almost surely as $n\to\infty$, there exists a linear subspace $V'_n$ of $V_{n}$ 
of dimension $l$ such that for any norm $1$ vector $x$ of $V'_n$, the 
singular values of $x$ are $\e$-close to $\lambda$ 
and the vectors $e_{i}$ appearing in the singular value decomposition \eqref{eq:vector-SVD} of $x$ are $\e$-close to the vectors 
of a fixed orthonormal basis of $\C^k$.
\end{proposition}

\begin{proof}
We prove this theorem by induction over $l$. For $l=1$, this is Theorem \ref{thm:lln-largest-eigenvector}. 
In the remainder of the proof, our standing assumption is that almost surely 
as $n\to\infty$, there exists a linear subspace $V'_n$ of $V_{n}$ 
of dimension $l$, spanned by $l$ eigenvectors of $P_{n}(A\otimes I_{n})P_{n}$, such that for any norm $1$ vector $x$ of $V'_n$, the 
singular values of $x$ are $\e$-close to $\lambda$ 
and the vectors $e_{i}$ appearing in the singular value decomposition of $x$ are $\e$-close to the vectors 
of a fixed orthonormal basis of $\C^k$. Since the singular value decomposition of vectors \eqref{eq:vector-SVD} is continuous in all of its parameters, 
we can assume that the subspace $V_n'$ is spanned by $l$ eigenvectors $y_1, \ldots, y_l$ which satisfy
$$\forall 1 \leq j \leq l, \quad y_j = \sum_{i=1}^k \sqrt{\lambda_i} e_i \otimes f_i^{(j)} + \tilde{y_j},$$
where $e_i$ is the aforementioned fixed basis of $\C^k$ and $\tilde y_j $ is a correction of small norm:
$$\|\tilde y_j \| \leq \e.$$

Our task is to find an additional vector $y_{l+1} \in V_{n} \setminus V'_{n}$ such that the vector space
$V''_{n}=\spn \{y_{l+1},V'_{n}\}$ satisfies almost surely, as $n\to\infty$, for any norm $1$ vector $x$ of $V''_n$, the singular values of $x$ are $\e$-close to $\lambda$ and the $\C^k$ part
of its singular vectors are close to the $e_i$. 
As stated before, we shall choose $y_{l+1}$ to be an eigenvector of $P_{n}(A\otimes I_{n})P_{n}$. This choice being made 
for $y_1, \ldots, y_l$, it ensures the orthogonality relation $y_{l+1} \perp V_n'$.
In view of Theorem \ref{thm:lln-largest-eigenvector}, for this strategy to work, we need to choose
$y_{l+1}$ an eigenvector corresponding to a large eigenvalue; this ensures that $y_{l+1}$ itself satisfies the singular value and 
singular vector requirements. We now need to show that
\emph{every} vector of $V''_{n}=\spn \{ y_{l+1},V'_{n}\} = \spn \{ y_1, \ldots, y_{l+1} \}$ satisfies the same requirements. 

In order to conclude, we need to chose an eigenvector $y_{l+1}$ which is orthogonal to all the vectors in the set 
$$Y = \{ e_{i_1} \otimes f_{i_2}^{(j)} \; | \; 1 \leq i_1, i_2 \leq k, 1 \leq j \leq l \}.$$
This can be done, since we may choose $y_{l+1}$ from a list of $\nu_n$ eigenvectors of $P_{n}(A\otimes I_{n})P_{n}$ (corresponding to the $\nu_n$ largest eigenvalues).

Indeed, start from the simple observation that the $\nu_{n}$ eigenvectors associated with the $\nu_{n}$ largest eigenvalues of $P_{n}(A\otimes I_{n})P_{n}$ (call them $x_{1},\ldots ,x_{\nu_{n}}$), are orthogonal, and therefore satisfy the following Parseval inequality: 
$$\sum_{i=1}^{\nu_{n}}|\scalar{{x_{i}}}{y}|^{2}\leq 1$$
for any vector $\|y\|\leq 1$.
Therefore it follows that there are at least $\nu_{n}-\varepsilon^{-1}$ of them that satisfy
$$|\scalar{{x_{i}}}{y}|^{2}\leq \varepsilon .$$
Similarly, let now $Y$ be a finite collection of norm $1$ vectors. The union bound tells that there are 
at least $\nu_{n}-|Y|\varepsilon^{-1}$ of them that satisfy
$$|\scalar{{x_{i}}}{y}|^{2}\leq \varepsilon $$
for any $y\in Y$. As soon as $\nu_{n}>kl/\varepsilon$, we are guaranteed the existence of an eigenvector $y_{l+1}$ which is almost orthogonal to all the terms appearing in the singular value decomposition of each of $y_1, \ldots, y_l$. This implies that, for all $1 \leq i_1, i_2 \leq ki$ and $1 \leq j \leq l$, 
$$|\scalar{f_{i_1}^{(l+1)}}{f_{i_2}^{(j)}}| \leq 2 \e.$$

Let us now consider an arbitrary norm one vector in $V_n'' = \spn\{y_1, \ldots, y_{l+1}\}$ and compute its (approximate) singular value decomposition. Let $(\alpha_1, \ldots, \alpha_{l+1})$ be a unit norm vector in $\C^{l+1}$. 
$$\sum_{j=1}^{l+1} \alpha_j y_j = \sum_{i=1}^k \sqrt{\lambda_i} e_i \otimes \left[ \sum_{j=1}^{l+1} \alpha_j f_i^{(j)} \right] + \sum_{j=1}^{l+1} \alpha_j \tilde y_j. $$

Since the vectors $\sum_{j=1}^{l+1} \alpha_j f_i^{(j)}$ form an orthogonal family for $1 \leq i \leq k$, it follows  that the conclusion of Proposition \ref{prop:lln-subspace} holds at dimension $l+1$ and with an appropriately updated value of the error term $\e$.

\end{proof}

We also need the following elementary lemma:
\begin{lemma}
\label{lem:continuity-hausdorff}
Let $F:\R^{p}\times\R^{q}\to\R^{p}$ be a continuous map such that $F(\cdot , 0)=\id_{p}$. Let $K$ be a subset of $\R^{p}$ and $K'$ be a compact subset of the interior of $K$. Then, there exists a neighborhood of $0$ in $\R^{q}$ such that for any $y$ in this neighborhood, $K'\subset F(K,y)$.
\end{lemma}

\begin{proof}
Since $K'$ is compact, without loss of generality we may assume that $K$ is bounded. The continuity assumption on $F$ and the boundedness of $K$ imply that the map $y \mapsto F(K,y)$ is continuous with respect to the Hausdorff distance. The result follows then readily from this observation.
\end{proof}

Finally we state a result that will complete the proof of Theorem \ref{thm:main}. 

\begin{theorem}
\label{thm:lower-bound}
For any compact set $\mathcal K$ contained in the interior of $K_{k,t}$, almost surely for $n$ large enough,  $\mathcal K \subset K_{n,k,t}$.
\end{theorem}

\begin{proof}
We shall prove a slightly stronger version of this result.
Let $\mathcal{P}_N$ be the subset of rank one selfadjoint projections of $End (V_n)$.
The inclusion $V_n\subset \C^k\otimes \C^n$ induces a non-unital inclusion of matrix algebras $End (V_n)\subset \M_k(\C)\otimes\M_n(\C)$.
Let $\hat K_{k,t}$ be the collection of self-adjoint matrices in $\M_k(\C)$ whose eigenvalues belong to $K_{k,t}$.
This is clearly a compact subset of $\M_k(\C)$, and it is of non-empty interior in the 
affine variety of trace one self-adjoint matrices.
(Indeed, for any $a\in\Delta_k$ which is not a multiple of the identity, $\langle1^k,a\rangle/k<\normt{a}$,
while when $a$ is a multiple of the identity, the inequality $\langle\lambda,a\rangle\le\normt{a}$
is trivially satisfied for all $\lambda\in\Delta_k$, so $\hat K_{k,t}$ contains a neighborhood of $1^k/k$).
If we can prove that for any compact subset  $\hat{\mathcal K}$ of the interior of $\hat K_{k,t}$, with probability one, for $n$ large enough, 
$\hat{\mathcal K}\subset  (id_k\otimes\trace_n)(\mathcal{P}_N)$,
then the theorem will be proved. One may think of this new problem as a quantum version of the original problem. 

So, let us concentrate on proving this fact. In order to simplify notation, let us denote $(id_k\otimes \trace_n)(\mathcal{P}_N)$
by $\hat K_{n,k,t}$. 
Since from any covering of a compact set by open sets one can extract a finite sub-covering, it is enough to prove that for any closed ball of center $x$ and radius $\e$ in the interior of $\hat K_{k,t}$, almost surely for $n$ large enough, $\overline{B(x,\e)}$ is contained in  the interior of $\hat K_{n,k,t}$. 

Given the closed ball $\overline{B(x,\e)}$, let 
$A_1,\ldots ,A_m$ 
be exposed points of $\hat K_{k,t}$ whose convex hull contains a neighborhood of $\overline{B(x,\e)}$. Such 
$A_1,\ldots, A_m$ always exist because the set of exposed points is dense in the set of extremal points, by a result of Straszewicz (\cite{rockafellar}, Theorem 18.6).

Let $y_i\in V_n$ be a norm one vector such that $A_i$ is the orthogonal rank one projection onto $\C y_i$.
For each $i\in \{1,\ldots ,m\}$, let $V'_{i}$ be a vector subspace 
of dimension $l$ (to be specified later) 
as in Proposition \ref{prop:lln-subspace}. Let $x_{1}\in V'_{1}$ be any norm $1$ vector and let $f_{i}^{(1)}$ be the vectors in $\C^{n}$ appearing 
in its singular value
decomposition. Using Proposition \ref{prop:lln-subspace} and making an appropriate by 
Gram-Schmidt procedure, since the dimension $l$ is large enough, we can find $x_{2}\in V'_{2}$ such that the 
vectors $f_{i}^{(2)} \in \C^{n}$ appearing in its Schmidt decomposition are all orthogonal to all $f_{i}^{(1)}$ $i\in \{1,\ldots ,k\}$.

By induction, we can find $x_{j}\in V'_{j}$ such that the vectors $f_{i}^{(j)} \in \C^{n}$ appearing in its Schmidt decomposition are all orthogonal to all $f_{i'}^{(j')}$, for all $i'\in \{1,\ldots ,k\}$ and $j'<j$.

For $n$ large enough, it follows from Lemma \ref{lem:continuity-hausdorff} (and from the fact that the use of Proposition \ref{prop:lln-subspace} ensures an appropriate convergence of the $e_{i}\in \C^{k}$ part of the Schmidt decomposition), that the collection of Schmidt vectors of a linear combination
$$\{\alpha_{1}x_{1}+\ldots +\alpha_{m}x_{m},\sum |\alpha_{i}|^{2}=m\}$$
contains $\overline{B(x,\e)}$.
\end{proof}

\begin{corollary}
In the metric space of compact subsets of $\Delta_k$ endowed with the Hausdorff distance, the distribution of $\partial K_{n,k,t}$ converges in probability to the Dirac mass 
on $\partial K_{k,t}$. 
\end{corollary}

\begin{proof}
It is enough to prove that the result holds almost surely. It follows from Theorem \ref{thm:main} that for any $\varepsilon >0$, with probability one, for $n$ large enough, $\partial K_{n,k,t}$ is 
included in a $\varepsilon$-neighborhood of
$\partial K_{k,t}$.

Let us prove the converse inclusion.
Let $x$ be an element in the interior of $K_{k,t}$ and  $y$ be an element in $\partial K_{k,t}$. 
Our results so far imply that, for $n$ large enough, $x$ is an element of $K_{n,k,t}$. Let $t_n\in \mathbb{R}_+$ be the maximal number such that
$x+t_n(y-x)\in K_{n,k,t}$. By the upper bound in Theorem \ref{thm:main}, we have $\limsup_n t_n\leq 1$. The strict inequality $\liminf_n t_n <1$ would yield a contradiction for the lower bound in the same theorem, therefore $\lim_n t_n=1$. This implies that $y$ is in a $\varepsilon$-neighborhood of $\partial K_{n,k,t}$.

Since this result holds true for all boundary points $y\in \partial K_{k,t}$, the proof is complete. 
\end{proof}

\section{Properties of the limiting set $K_{k,t}$ and of its dual}\label{sec:dual-ball}

In this final section we derive geometric and convexity-related properties of the set $K_{k,t}$. Since this limiting set is described via the duality equation \eqref{eq:convex}, we start by investigating the unit ball of the $(t)$-norm. The reader might find it helpful to think as $K_{k,t}$ as the intersection of the dual of a ``ball'' formed by gluing two cones along their bases (a cylinder) with the probability simplex $\Delta_k$. The two vertices correspond to the upper and lower discs of the cylinder, 
the points on the circle along which the cones are glued correspond to vertical segments on the vertical wall of the cylinder, while the points of the two ``circles'' bordering the upper and lower discs of the cylinder are the images of segments starting from the two vertices of the cones.

\subsection{Preliminary observations}

Using the permutation invariance of the $\normt{\cdot}$ norm, it is clear that $K_{k,t}$ is invariant under permutation of coordinates. 
We start with the following lemma:

\begin{lemma}
Let $C$ be the interior of the Weyl chamber $\Delta_k^\downarrow$ of the probability simplex. Let $\lambda \in C$ be an exposed point of $K_{k,t}$ and $a \in \Delta_k$ a direction such that $H(a, t) \cap K_{k,t} = \{\lambda\}$. Then $a \in C$.
\end{lemma}
\begin{proof}
First, let us show that $a \in \Delta_k^\downarrow = \bar C$. If this would not be the case, then there exists a direction $a' \in \Delta_k^\downarrow$, obtained by permuting the coordinates of $a$, such that
$$\scalar{\lambda}{a'} > \scalar{\lambda}{a}.$$ 
From this, we deduce that $\scalar{\lambda}{a'} > \scalar{\lambda}{a} = \normt{a} = \normt{a'}$, hence $\lambda \notin H^+(a', t)$, which contradicts the fact that $\lambda \in K_{k,t}$. 

Next, let us show that $a$ is not degenerate, 
i.e. it has distinct coordinates. Should $a$ have two equal coordinates, say the $i$-th and the $j$-th, let $\lambda' \in K_{k,t}$ be the vector obtained by permuting the $i$-th and the $j$-th coordinates in $\lambda$. As before, it follows that $\scalar{\lambda}{a} = \scalar{\lambda'}{a}$ and thus $\{\lambda, \lambda'\} \subset H(a, t) \cap K_{k,t}$ which is a contradiction.
\end{proof}

The following proposition shows that, in a certain sense, the $t$-norm interpolates between $\ell^{1}$ and $\ell^{\infty}$ norms when $t\in (0,1]$ and $x\in \mathbb R^k_+.$

\begin{proposition} 
For any $x \in \R^k$, $\|x\|_{(t=1)} = \|x\|_\iy$ and $\lim_{t \to 0^+} \normt{x} = k^{-1} | \sum_{i=1}^k x_i |$.
\end{proposition}

\begin{proof}
The first statement is just a re-phrasing of the definition of $\normt{x}$ at $t=1$. The second is a re-phrasing of the free law of large numbers: 
as we know from the superconvergence result of Bercovici and 
Voiculescu \cite{BV}, 
if $X_1,X_2,\ldots$ are free i.d. random variables, centered at $a$ and
with variance $\sigma^2$, then 
$$
\mu_\frac{X_1-a+X_2-a+\cdots+X_N-a}{\sqrt{N}}=
\mu_{\frac{X_1+X_2+\cdots+X_N}{\sqrt{N}}-a\sqrt{N}}\to\frac{1}{2\pi\sigma^2}
\sqrt{4\sigma^2 - u^2}\ind_{(-2\sigma, 2\sigma)}(u)\,\d{u}
$$
in the sense that the ends of the supports of $\mu_{\frac{X_1+X_2+\cdots+X_N}{\sqrt{N}}-a\sqrt{N}}$ converge to $\pm 2 \sigma$.
Taking $t=1/N$, $N \to \iy$, contraction by $1/N$ of $X_1+\cdots+X_N$ corresponds to taking
$\mu_{\frac{X_1-a+X_2-a+\cdots+X_N-a}{N}+a}=
\mu_{\frac{1}{\sqrt{N}}\cdot(\frac{X_1+X_2+\cdots+X_N}{\sqrt{N}}-a\sqrt{N})+a}=\mu_{\frac{1}{\sqrt{N}}\cdot(\frac{X_1+X_2+\cdots+X_N}{\sqrt{N}}-a\sqrt{N})}\boxplus\delta_a$. Then these measures converge to 
$\delta_a$ in the sense that the ends of the support converge to $a$.
We obtain our result by taking $X_1$ to be distributed according to $\mu_x$, in which case $a=k^{-1} \sum_{i=1}^k x_i$.
\end{proof}

In \cite{collins-nechita-2}, using similar ideas, it was shown that the set $K_{k,t}$ is included in the convex polytope $L_{k,t}$ defined by the following sequence of linear inequalities:

\begin{equation}
	\scalar{x}{y_j} \leq \normt{y_j} \quad \text{ where } y_j = 1^j0^{k-j} \text{ for } j=1,2, \ldots, k.
\end{equation}

This polytope was shown to be closely related to the majorization relation ``$\prec$'' \cite{bhatia}. Actually, in \cite{collins-nechita-2}, it was shown that $L_{k,t} = \{x \in \Delta_k \; | \; x \prec \beta^{(t)} \}$ where 
\begin{equation}
\label{def-beta-t}
\beta^{(t)}_j = \normt{1^j0^{k-j}} - \normt{1^{j-1}0^{k-j+1}}, \quad \forall \; 1 \leq j \leq k.
\end{equation}
However, the inclusion $K_{k,t} \subset L_{k,t}$ is strict, since $K_{k,t}$ is defined by a larger set of inequalities,
and most of the inequalities are not redundant, as it is shown in the next section.

\subsection{Study of the geometry of $K_{k,t}$ and of the unit ball of the $(t)$-norm}

Next we shall remind the reader of a few elementary convex analysis results. First, the correspondence $\mathbb R^k\ni u\mapsto H_u=\{x:\langle u,x\rangle=1\}$ is a bijection between vectors and hyperplanes in $\mathbb R^k$. If $A$ is a compact convex set whose interior contains the origin of $\mathbb R^k$, we shall denote by $A^*$ its {\em polar dual} (or, for short, dual), i.e. $A^*=\{x\in\mathbb R^k:\langle x,a\rangle\le1\textrm{ for all }a\in A\}$. 
An {\em exposed face} of $A$ is a set $A\cap H_u$ for some hyperplane $H_u$ with the property that $\langle a,u\rangle\leq1$ for all $a\in A$. For any given exposed face $B$ of $A$, we can define the {\em polar
face mapping of $A$}
$$
\varphi(B)=\{x\in A^*:\langle b,x\rangle=1\textrm{ for }b\in B\}.
$$
Then \cite[Theorem 2.8.6]{webster} $\varphi$ is an inclusion reversing bijection. Moreover, if $b_0$ belongs to the relative interior of $B$, then $\varphi(B)=\{x\in A^*:\langle b_0,x\rangle=1\}$ \cite[Exercise 2.8.4]{webster}. We shall study this correspondence in more detail for the
case when $A$ is the unit ball of a norm (eventually of $\normt{\cdot}
$). 

We note that for a given arbitrary norm $\|\cdot\|$, the boundary of the unit ball $\partial\{x\in\mathbb R^k\colon\|x\|=1\}$ is a $k-1$-dimensional
topological manifold, which admits projections as atlases. Indeed, let
$x_0\in\mathbb R$ so that $\|x_0\|=1$. We claim that the projection 
onto $\{x_0\}^\perp$ of the set $\{x\in\mathbb R^k\colon\|x\|=1,
\norm{x-\scalar{x}{x_0} x_0}<1, \scalar{x}{x_0} > 0\}$ is a continuous bijection with continuous inverse.
First, continuity is clear. Next, pick $b\in\{x_0\}^\perp$ with $\|b\|<1$, and consider $b+tx_0$, $t\in\mathbb R$. Then $\|b+tx_0\|\ge|\|b\|-|t|\|x_0\||$ so there must be points $t$ so that 
$\|b+tx_0\|=1$. Convexity guarantees that there are either two such
points, or exactly one continuum of them. The second possibility is
easily discarded, since there must be both positive and negative such
numbers, and at $t=0$ the inequality is strict. Also, only one
of those two points satisfies $\langle b+tx_0,x_0\rangle>0$, as
$b\perp x_0$. Thus we have identified our bijection. Clearly a 
proper continuous bijection is a homeomorphism, so our claim is proved.

Let us remind the reader the notion of gradients and subgradients. First, for a convex function $f$ we define the \emph{one-sided} directional derivatives of $f$ at $x$ relative to $y$ by
$$
f'_+(x;y)=\lim_{\lambda\to0^+}\frac{f(x+\lambda y)-f(x)}{\lambda},\quad
f'_-(x;y)=\lim_{\lambda\to0^-}\frac{f(x+\lambda y)-f(x)}{\lambda}.
$$
It is easy to observe that $-f'_+(x;-y)=f'_-(x;y)$, so that the 
directional derivative at $x$ in the direction $y$ exists if and only 
if the one-sided directional derivatives exist and satisfy the relation $f'_+(x;y)=
-f'_+(x;-y).$ The inequality $f'_+(x;y)\ge f'_-(x;y)$ holds, 
and generally $f'_+(x;\cdot)$ is a positively homogeneous convex 
function on $\mathbb R^{k}$ for any $x$. If $f'(x;\cdot)$ exists, then 
it is linear \cite[Theorem 5.5.2]{webster}. 

The gradient of $f$ at $x$ (if existing) is defined as 
$
\nabla f(x)=\left(\partial_1f(x),\dots,\partial_kf(x)\right),
$
where we use the short-hand notation $\partial_jf=\frac{\partial f}{\partial x_j}$. This means 
$$
\langle\nabla f(x),y\rangle=\sum_{j=1}^ky_j\partial_jf(x)=f'(x;y).
$$
We observe that, generally, for a norm we have 
$\left|\frac{\|x+\lambda y\|-\|x\|}{\lambda}\right|
\leq\|y\|$ so (by a slight abuse of 
notation) we can write for our specific case $f'(x;y)=[f'_-(x;y),
f'_+(x;y)]\subseteq[-f(y),f(y)]=[-\norm{y},\norm{y}].$

A subgradient of a convex function $f$ at a point $x$ is a vector 
$x^*\in\mathbb R^k$ so that 
$$
f(y)-f(x)\ge\langle x^*,y-x\rangle,\quad\forall y\in\mathbb R^k.
$$
(For our case, $\|y\|-\|x\|\ge\langle x^*,y-x\rangle$.) Geometrically, 
this means that $h(y)=f(x)+\langle x^*,y-x\rangle$ is a nonvertical
supporting hyperplane of the epigraph of $f$ at the point $(x,f(x))$
\cite[Section 23]{rockafellar}. The set of all subgradients of $f$ 
at $x$ is called the {\em subdifferential} of $f$ at $x$ and is 
denoted by $\partial f(x)$. If $f$ is differentiable, then
$x^*$ is unique and $x^*=\nabla f(x),$ and, conversely, if $\partial
f(x)$ contains exactly one point, then $f$ is differentiable at $x$
\cite[Theorem 25.1]{rockafellar}.

In addition, if the correspondence $x\mapsto\|x\|$ is differentiable 
around a point $a\neq0$ then the atlas described above is differentiable around $a$. Indeed, let us assume $x\mapsto\|x\|$ is 
differentiable at $a$. It is clear that the derivative of this map in
the direction $a$ at $a$ equals $\|a\|$, so $a$ is not a singular 
point. For $x\in A$ close enough to $a$ its image in $\{a\}^\perp$ is 
$b=x-\frac{\langle x,a\rangle}{\langle a,a\rangle}a$. So the 
correspondence from $b$ to $x$ is given by an implicit equation:
$x=b+ta$, $t\ge0$. Then we write the implicit function 
equation for $\mathcal F(b,t)=\|b+ta\|$ as $\mathcal F(b,\mathfrak k(b))=1$. As we know of the existence of the solution $\mathfrak k(b)$,
we only need to verify differentiability: $\partial_t\mathcal F(b,t)=
\langle\nabla\|(b+ta)\|,a\rangle$ in $t=\mathfrak k(b)$ is 
well-defined by hypothesis and nonzero by the condition that $b+\mathfrak k(b)a$ is close to $a$ (we know that $\langle\nabla\|a\|,a
\rangle\geq\|a\|=1$ from the subgradient inequality above evaluated in
$x=a$ and $y=0$). 

The above considerations will allow us to to perform a geometric analysis of the ball of the $(t)$-norm and its dual.

Let us now analyze the correspondence between faces in terms of their 
dimensions. 
The general result which is of interest for us will be stated in the following remark:

\begin{remark}\label{smoothness}
Assume that $g(x)$ is a norm so that $g^{-1}(1)$ is the real part of an analytic set in the sense of \cite{chirka}. 
Denote by $A=\{x\in\mathbb R^k\colon g(x)\leq1\}$, and $A^*$ the unit ball in 
the dual norm. We define $\varphi$ to be the polar face map from the 
faces of $A$ to the faces of $A^*$. Then
\begin{enumerate}
\item If $x\in\partial A$ is a point belonging to the relative interior of an exposed face $B$ of $A$ so that $\partial A$ is a smooth manifold
around $x$, then $\varphi(B)$ is a point in $\partial A^*$;
\item If $x\in\partial A$ is a point belonging to the relative interior 
of an exposed face $B$ of $A$ where there are $j\in\{1,\dots,k-1\}$ 
independent 
directions in which $g$ is not differentiable, then $\varphi(B)$ has
dimension $j$.
\end{enumerate}
\end{remark}

In particular, an isolated ``vertex'' of such a ball, where the norm function is not differentiable in any direction different from the vertex, 
corresponds to a piece of hyperplane having nonempty $k-1$-dimensional interior, an ``edge'' - a segment included in the $t$-sphere 
determining only one direction of differentiability - corresponds via $\varphi$ to a $k-2$-dimensional piece and so on. 
The case important for us is when the unit ball
is an analytic set (in the sense of \cite{chirka}), so its points
of non-smoothness are well understood in terms of dimension.

\begin{proof}

Fix a point $x_0$ with $g(x_0)=1$ and let $B$ be the face in whose 
relative interior $x_0$ lives. 
Recall that $\varphi(B)=\{x\in\mathbb R^k
\colon\langle x,x_0\rangle=1,\langle x,a\rangle\le1\forall a\in A\}=
\{x\in\mathbb R^k\colon\langle x,x_0\rangle=g(x_0),\langle x,a\rangle\le g(a)\forall a\in\mathbb R^k\}.$ Subtracting the two defining 
relations $g(a)\ge\langle x,a\rangle, g(x_0)=\langle x,x_0\rangle$
from each other gives $g(a)-g(x_0)\ge\langle x,a-x_0\rangle$
This indicates that $x\in\varphi(B)\implies x\in\partial g(x_0)$, i.e.
$$
\varphi(B)\subseteq\partial g(x_0).
$$
In particular, if $g$ is differentiable in $x_0$, then $\varphi(B)$
contains exactly one point, as claimed in (1).

We note however that evaluating $g(a)-g(x_0)\ge\langle x,a-x_0\rangle$
in $a=tx_0$ gives $(t-1)g(x_0)\ge(t-1)\langle x,x_0\rangle.$
In particular, when $t=0,$ we obtain $-g(x_0)\ge-\langle x,x_0\rangle$,
i.e. $g(x_0)\le\langle x,x_0\rangle$, and when $t=2$ we obtain 
$g(x_0)\ge\langle x,x_0\rangle$. Thus, $g(x_0)=\langle x,x_0\rangle$. 
Also, for $a=b+x_0$ we have 
$g(b)\ge g(b+x_0)-g(x_0)\ge\langle x,b\rangle$ for all $b\in\mathbb R^k
$. So $\partial g(x_0)\subseteq\varphi(B).$ Thus,
\begin{equation}\label{subgradient-face}
\varphi(B)=\partial g(x_0)\quad\forall x_0\textrm{ in the relative 
interior of the exposed face }B.
\end{equation}

Generally, from the definition of $\varphi(B)$ it follows that
$x\in\varphi(B)$
if and only if $a\mapsto{g(a)}-{\langle x,a\rangle}$ reaches a 
global minimum at $a=x_0$ on all of $\mathbb R^k$.
In particular, we look at $a=x_0+\lambda y_0$. Differentiation with
respect to $\lambda$ to left and right of zero gives $g'_\mp(x_0;
y_0)-\langle x,y_0\rangle$. As $x_0$ is a point of minimum, it is 
clear
that $\lambda\mapsto{g(x_0+\lambda y_0)}-{\langle x,x_0+\lambda y_0
\rangle}$ must decrease as $\lambda$ grows to zero, and then increase
after $\lambda$ passed the point zero. So the derivative must 
either be zero or change sign at $\lambda=0$. So $g'_-(x_0;
y_0)-\langle x,y_0\rangle\le0,g'_+(x_0;
y_0)-\langle x,y_0\rangle\ge0$, i.e. $\langle x,y_0\rangle\in
[g'_-(x_0;y_0),g'_+(x_0;y_0)]$. As $g_\pm(x_0;\cdot)$ is positively 
homogeneous, we may assume $g(y_0)=1$. Thus, we can write as a 
condition for $x\in\varphi(B)$
$$
x\in\varphi(B)\implies \langle x,y_0\rangle\in[g'_-(x_0;y_0),
g'_+(x_0;y_0)]\textrm{ for all } y_0\in\mathbb R^k,g(y_0)=1,
$$
which means that 
\begin{equation}\label{condition}
\partial g(x_0)\subseteq\{x\in\mathbb R^k\colon
g'_-(x_0;y_0)\le\langle x,y_0\rangle\le g'_+(x_0;y_0)\forall y_0\in
\partial A\}.
\end{equation}

Let us note that if there are $l$ 
linearly independent directions $y_1,\dots,y_l$ 
in $\{x_0\}^\perp$ so that $g$ is
differentiable in all these directions at $x_0$, then for any
vector $z\in\mathrm{Span}\{y_1,
\dots,y_l,x_0\}\subset\mathbb R^k$, $g'(x_0;z)$ exists. Indeed, 
the function $\mathrm{Span}\{y_1,
\dots,y_l,x_0\}\ni z\mapsto g(x_0+z)$ is still convex. The partial 
derivatives of this function in zero, $\lim_{t\to0}\frac{g(x_0+ty_i)-
g(x_0)}{t}$, $i\in\{1,2,\dots,l\}$
and $\lim_{t\to0}\frac{g(x_0+tx_0)-g(x_0)}{t}$ all exist,
so the function $z\mapsto g'_+(x_0;z)$ satisfies 
$ -g'_+(x_0;z)= g'_+(x_0;-z)$ for $z\in\{x_0,y_1,\dots,y_l\}$. Since 
$z\mapsto g'_+(x_0;z)$ is positively homogeneous and convex 
\cite[Theorem 23.1]{rockafellar}, it follows from 
\cite[Theorem 4.8]{rockafellar} that 
$z\mapsto g'_+(x_0;z)$ is in fact linear on $\mathrm{Span}\{x_0,y_1,\dots,y_l\}.$ This, according to \cite[Theorem 25.2]{rockafellar},
implies that $g'_+(x_0;\cdot)$ is differentiable on 
$\mathrm{Span}\{x_0,y_1,\dots,y_l\}.$
 Thus, $g'(x_0;z)=\lim_{t\to0}\frac{g(x_0+tz)-
g(x_0)}{t}$ exists for any $z\in\mathrm{Span}\{y_1,
\dots,y_l,x_0\}.$ This indicates that whenever $z\in\mathrm{Span}\{y_1,
\dots,y_l,x_0\}$ and $x\in\varphi(B)$, 
$\langle x,z\rangle=g'(x_0;z)$. This gives us a system of 
$l+1$ equations with $k$ unknowns, so it specifies for $x$ exactly
$l+1$ degrees of freedom. So $\partial g(x_0)$ is contained in an 
affine variety of dimension at most $k-(l+1)$.

To complete the proof we only need to show that for any of the {\em
other} $k-(l+1)$ directions, $x\in\varphi(B)$ is free to move
for a nonzero distance, i.e. that $\varphi(B)$ is open in the $k-(l+1)$-dimensional affine variety in which it lives. 
First of all, we must note that for any $w\not
\in\mathrm{Span}\{y_1,\dots,y_l,x_0\}$, $g'(x_0;w)$ does not exist.
Indeed, by \cite[Theorem 4.8]{rockafellar}, any positively homogeneous
convex function $f$ is linear on a subspace $L$ if and only if $f(-x)=
-f(x)$ for all $x\in L$, and this condition is true if merely $f(-b_i)=
-f(b_i)$ for all $b_1,\dots,b_m$ forming a basis (not necessarily
orthogonal!) of $L$. Applying this as above to the right derivative 
$g'_+(x_0;\cdot)$ we conclude that if $g_+'(x_0;\cdot) $ is 
differentiable on the higher dimensional space
$\mathrm{Span}\{x_0,y_1,\dots,y_l,w\}$, a contradiction.
We know \cite[Section 23]{rockafellar} that $\partial g(x_0)$
is closed and convex, so assume that $x$ is in the relative interior 
of $\partial g(x_0)$. 
Choose any direction $z\perp\mathrm{Span}\{x_0,y_1,\dots,y_l\}$.
We claim that for $|t|$ small enough, $x+tz\in\partial g(x_0)$.
Indeed, this is equivalent to the statement that 
$g(x_0+b)-g(x_0)-\langle x+tz,b\rangle\ge0$ for all $b\in\mathbb R^k$. 
As $\Phi:b\mapsto g(x_0+b)-g(x_0)-\langle x+tz,b\rangle$ takes the
value zero in $b=0$, we would like to show that $b=0$ is a point of
global minimum. In particular, we shall take the real function 
$\mathbb R\ni\lambda\mapsto\Phi(\lambda b)$ and we shall decompose
$b=b_s+b_p$ with $b_s\in\mathrm{Span}\{x_0,y_1,\dots,y_l\}$ and
$b_p\perp\mathrm{Span}\{x_0,y_1,\dots,y_l\}$, and, in particular,
$\langle b_s,z\rangle=0$.
We have
$$
{\Phi(\lambda b)}=g(x_0+\lambda b)-g(x_0)
-\lambda\langle x+tz,b\rangle=g(x_0+\lambda b)-g(x_0)-\lambda
\langle x,b\rangle-t\lambda\langle z,b_p\rangle.
$$
Differentiating in $\lambda$ gives 
$g_\pm'(x_0+\lambda b)-\langle x,b\rangle-t\langle z,b_p\rangle.$ (We
have used $\pm$ to denote that we consider, in the points where the
derivative does not exist, the right and left derivatives; it is
known that, $\lambda\mapsto g(x_0+\lambda b)$ being convex, these
two exist and $g_-'(x_0+\lambda b)\le g_+'(x_0+\lambda b)$.)
Thus, as function of $\lambda$, we can state that
$g_\pm'(x_0+\lambda b)-\langle x,b\rangle-t\langle z,b_p\rangle$
is strictly increasing, with jump increases at the points of 
non-differentiability. In zero, by hypothesis 
$g_-(x_0;b)<g_+(x_0;b)$
and $g_-(x_0;b)\le\langle x,b\rangle\le g_+(x_0;b)$ for all
$b\in\mathbb R^k$ (see \eqref{condition}). As $x$ is in the
relative interior of $\varphi(B)$, we have 
$g_-(x_0;b)<\langle x,b\rangle<g_+(x_0;b)$ for all
$b\in\mathbb R^k$. We assume now that $g(b)=1$.
Then clearly for $|t|$ small enough, 
$g_-(x_0;b)<\langle x,b\rangle+t\langle z,b_p\rangle<g_+(x_0;b)$
holds. Since both $g_\pm'(x_0;\cdot)$ are positively
homogeneous, this is equivalent to 
$g_-(x_0;hb)<\langle x,hb\rangle+t\langle z,hb_p\rangle<g_+(x_0;hb)$
for all $h>0$. Thus, 
$\lambda\mapsto
g_\pm'(x_0+\lambda b)-\langle x,b\rangle-t\langle z,b_p\rangle$
changes sign exactly at $\lambda=0$. This proves our statement.
\end{proof}

We shall apply these simple observations in a corollary to the following theorem, which describes the unit ball of the norm $\normt{\cdot}$ (for a picture in the case $k=2$, see Figure \ref{fig:k2}).

\begin{theorem}\label{convex-norm}
The boundary of the unit ball in the norm $(t),$ denoted $S_t$, is locally analytic. It can be expressed as the union of two intersecting cones, one with vertex at $1^k$, and the other with vertex at $(-1)^k$. Its points of non-analyticity are as follows:
\begin{itemize}
\item When $1-\frac{j}k< t<1-\frac{j-1}k$, then $S_t$ contains 
exposed faces of maximum dimension $k-j$;
\item In particular, when $t<\frac1k$, then $S_t$ contains no other segments except the ones connecting each point of $S_t$ either with 
$1^k$ or with $(-1)^k$, while if $\frac{k-1}k\le t$, then $S_t$ is 
simply the boundary of the unit ball in the $\ell^\infty$ norm on 
$\mathbb R^k$.
\end{itemize}
If $\normt{x}=t\min{\mathrm{supp}(\mu_x^{\boxplus 1/t})}$, then $x$ belongs to the cone with vertex at $(-1)^k$, and if $\normt{x}=t\max{\mathrm{supp}(\mu_x^{\boxplus 1/t})}$, then $x$ belongs to the cone with vertex at $1^k$.
Moreover, if $t<\frac1k$, then $\|\nabla\normt{b}\|_1=1$ for all 
 $b\in\R^k_+$, $b\not\in\mathbb R\cdot1^k$.
\end{theorem}
The above theorem tells us also that whenever $t<\frac1k$, the norm $(t)$ is ``one segment away'' from being strictly convex.

\begin{proof}
With the notation $t=1/s$, let us start by describing the set
$$
\{b\in\mathbb R^k\colon\max{\mathrm{supp}(\mu_b^{\boxplus 1/t})}\leq1\}=\{b\in\mathbb R^k_+\colon\max{\mathrm{supp}(\mu_b^{\boxplus s})}\leq1\}.
$$

To start with, we shall argue that
$
\{b\in\Rnc^k \colon\max\textrm{supp}(\mu_b^{\boxplus s})=1\}$ is 
an analytic set whenever $t<\frac1k$ or, equivalently, $s>k$. (We 
understand this to mean that this set is part of a larger complex 
analytic set in the sense of \cite{chirka}.) Observe that we can view 
$F_{\mu_b}(z)$ as a function of $k+1$ complex variables:
$$
F(b_1,\dots,b_k,z)=F_{\mu_b}(z)=k
\left[\frac{1}{z-b_1}+\frac{1}{z-b_2}+\cdots+\frac{1}{z-b_k}
\right]^{-1},
$$
for all $z\neq b_j$ so that 
$\frac{1}{z-b_1}+\frac{1}{z-b_2}+\cdots+\frac{1}{z-b_k}
\neq0$.
We record for future reference:
\begin{equation}\label{F'}
\partial_zF_{\mu_b}(z)
=\frac1kF_{\mu_b}(z)^2
\left[\frac{1}{(z-b_1)^2}+\cdots+\frac{1}{(z-b_k)^2}
\right],\quad
\partial_{b_j}F_{\mu_b}(z)
=-\frac1kF_{\mu_b}(z)^2
\frac{1}{(z-b_j)^2}.
\end{equation}
In particular,
\begin{equation}\label{nabla}
\partial_zF_{\mu_b}(z)=-\sum_{j=1}^k\partial_{b_j}F_{\mu_b}(z).
\end{equation}
Equation \eqref{atoms} guarantees that under our hypothesis $(\mu_b)^{\boxplus1/t}$ has no atoms,
so by Proposition \ref{53}, the supremum of the support of $(\mu_b)^{\boxplus1/t}$
is given by the largest real solution $w$ to the equation 
$(\partial_zF_{\mu_b})(w)=\frac{s}{s-1}$ via the formula 
$w+(\frac1s-1)F_{\mu_b}(w)$. We denote first by $w=f(b_1,\dots,b_k;s)$ 
the solution of $\partial_zF_{\mu_b}(w)=\frac{s}{s-1}$.
Our first claim is that the correspondence $(b_1,\dots,b_k;s)\mapsto
f(b_1,\dots,b_n;s)$ is analytic in a neighborhood of $(\mathbb R^k
\setminus\{(b,\dots,b)|b\in\mathbb
R\})\times(k,+\infty)$ {\em in}
$(\mathbb C^k\setminus\{(b,\dots,b)|b\in\mathbb C\})\times\mathbb C$. 
This follows directly from the implicit function theorem; to prove this, we 
shall rather write the partial derivatives of $f$ (for future reference) instead of just verifying the required 
conditions for $F$. So
\begin{eqnarray}
\partial_{b_j}f(b_1,\dots,b_k;s) & = & -\frac{(\partial_{b_j}\partial_zF)(b_1,\dots,b_k,f(b_1,\dots,b_k;s))}{
(\partial_z^2F)(b_1,\dots,b_k;f(b_1,\dots,b_k;s))};\\
\partial_{s}f(b_1,\dots,b_k;s) & = & -\frac{1}{
(\partial_z^2F)(b_1,\dots,b_k,f(b_1,\dots,b_k;s))(s-1)^2}.
\end{eqnarray}
We have seen from 
Proposition \ref{53} that, as the function $w\mapsto F_{\mu_b}(w)$ is strictly concave on the (unique) unbounded interval $J$ of analyticity containing arbitrarily large positive numbers, for any solution $f(b_1,\dots,b_k;s)\in J$ in vectors
$(b_1,\dots,b_k;s)\neq(b,\dots,b;s)$ (meaning away from the diagonal of $\mathbb R^k$), the function
$(\partial_z^2F)(b_1,\dots,b_k;f(b_1,\dots,b_k;s))\neq0$, so we easily conclude from the analyticity
of $\partial_zF$ that $f$ is complex analytic around these points viewed as points in $(\mathbb C^k\setminus\{(b,\dots,b)|b\in\mathbb C\})\times\mathbb C$. The easily observed fact that $F(b,\dots,b,z)=z-b
$ implies immediately that $f$ is not analytic in the variable $s$ in points $(b,\dots,b;s).$
In addition, the above together with Proposition \ref{53} implies that $f$ is not aanalytic in any of the other variables either in the points $(b,\dots,b)$.

The above equalities together with equation \eqref{nabla} yield
\begin{equation}\label{const}
\sum_{j=1}^k\partial_{b_j}f(b_1,\dots,b_k;s)=1.
\end{equation}
The expression for $\|b\|_{(1/s)}$ (or, more precise, for $t\max\textrm{supp}(\mu_b^{\boxplus s})$) is now written as
$$
f(b_1,\dots,b_k;s)+\left(\frac1s-1\right)F(b_1,\dots,b_k,f(b_1,\dots,b_k;s)).
$$
Differentiating this function in each coordinate $b_j$ gives
\begin{eqnarray*}
t\partial_{b_j}
\max\textrm{supp}(\mu_b^{\boxplus s}) & = & \partial_{b_j}f(b;s)+\left(\frac1s-1\right)
\left[(\partial_{b_j}F)(b,f(b;s))+(\partial_{z}F)(b,f(b;s))
\partial_{b_j}f(b;s)\right]\\
& = & \left(\frac1s-1\right)
(\partial_{b_j}F)(b,f(b;s)).
\end{eqnarray*}
(We have used here that $(\partial_{z}F)(b,f(b;s))=\frac{s}{s-1}$.)
This guarantees analyticity of the complex correspondence 
$b\mapsto\|b\|_{(1/s)}$ on a complex neighbourhood of the whole set $b\in\mathbb R^k$ on which the norm $\normt{\cdot}$ is achieved on the
upper bound of the support of $\mu_b^{\boxplus s}$, for $s>k$ fixed.
It is also remarkable that 
\begin{equation}\label{const-nabla}
\|\nabla\|b\|_{1/s}\|_1=\left(\frac1s-1\right)\sum_{j=1}^k(\partial_{b_j}F)(b,f(b;s))=-\left(\frac1s-1\right)
(\partial_zF)(b,f(b;s))=1,
\end{equation}
as $(\partial_{b_j}F)(b,f(b;s))$ is easily seen to be negative from \eqref{F'}.

We have proved now that the set $ \{b\in\Rnc^k\colon\max\textrm{supp}(\mu_b^{\boxplus s})
=1\}$ is the real part of an analytic set of complex dimension $k-1$ in $\mathbb C^k$.
We claim that this set cannot contain a line that does not contain $1^k$. Indeed, assume towards contradiction that there
exist $b,c\in\Rnc^k$ with 
$\max\textrm{supp}(\mu_b^{\boxplus s})=\max\textrm{supp}(\mu_c^{\boxplus s})=1$ so that 
$\max\textrm{supp}(\mu_{ub+(1-u)c}^{\boxplus s})\subset\{b\in\Rnc^k\colon\max\textrm{supp}(\mu_b^{\boxplus s})
=1\}$
for all $u\in[0,1]$.
Then, of course, $\max\textrm{supp}(\mu_{ub+(1-u)c}^{\boxplus s})\subset\{b\in\mathbb C^k\colon\max\textrm{supp}(\mu_b^{\boxplus s})
=1\}$
for all $u\in\mathbb R$ for which $\max\textrm{supp}(\mu_{ub+(1-u)c}^{\boxplus s})$ is well defined, i.e. for 
all $u\in\mathbb R$. 
However, the set $\{ub+(1-u)c\colon u\in\mathbb R\}$
must remain included in $\mathbb R^k$. This tells us that the upper bound of the support of 
$\mu_{ub+(1-u)c}^{\boxplus s}$ must remain equal to one for all $u\in\mathbb R$.
This is not possible: since $b\neq c$ (and, moreover, the two do not differ by a multiple of $1^k$) as $u$ tends to $\pm\infty$ clearly the diameter of the support of $\mu_{ub+(1-u)c}$ will tend to infinity. 
If the expectation of $\mu_{ub+(1-u)c}$ is nonconstant (as a function of $u$), then letting $u$ tend to infinity in the appropriate direction, we may make this expectation tend to plus infinity. Clearly,
as the expectation of $\mu_{ub+(1-u)c}^{\boxplus s}$ is simply $s$ 
times the expectation of $\mu_{ub+(1-u)c}$, we obtain a contradiction with the upper boundedness of the support of $\mu_{ub+(1-u)c}^{\boxplus
s}$. If the expectation of $\mu_{ub+(1-u)c}$ is a constant function of 
$u$, then $\sum b_j=\sum c_j$. Since $b\neq c$, there must be at least two distinct coordinates with differences of opposite signs, so when 
$|u|\to\infty$, both ends of the support of $\mu_{ub+(1-u)c}$ must
tend to infinity. Thus, the variance of $\mu_{ub+(1-u)c}$ will necessarily tend to infinity. Since the variance depends linearly of $
s$, it follows that the variance of $\mu_{ub+(1-u)c}^{\boxplus s}$
also tends to infinity. But this is impossible if the upper bound of
its support is constantly equal to one and at the same time its 
first moment stays constant.

This provided us the proof of the more difficult part of our theorem.
We note next that at times $t=j/k$, $j\in\{1,2,\dots,k\}$, we witness
certain ``phase transitions.'' Indeed, whenever $t\in(1-j/k,1-(j-1)/k)$
for some positive integer $j\le k$, Proposition \ref{53} part (2) and 
equation \eqref{atoms} guarantee that points of the form 
$(b_1,\dots,b_{k-j},w,\dots,w)$ with $-w<b_1,\dots,b_{k-j}\le w$ will 
have norm $(t)$ constantly equal to $1$. However, smaller atoms will 
disappear, i.e. if more than $k-j$ elements are of absolute value 
strictly less than $w$, the norm of this vector will be strictly 
smaller than $1$. Thus, these points will generate a set (in fact an 
exposed face) of dimension at most $k-j$ in the boundary of the unit
ball of radius one in $\normt{\cdot}$. This, in particular, guarantees
that for $t\ge\frac{k-1}{k}$, $\|\cdot\|_{(t)}=\|\cdot\|_\infty$.

Finally, the geometry of this ball as the intersection of two cones
is an immediate consequence of Proposition \ref{prop:t-norm-additive-conv}.
\end{proof}

The above theorem will allow us to draw some conclusions about the 
shape of the dual unit ball. We shall denote by $C^+$ and 
$C^-$ the two closed cones with vertex at $1^k$ and $(-1)^k$ 
respectively, so that $S_t = \{x\in\mathbb R^k\colon \normt{x}=1\}=C^+\cup 
C^-$. Note that for $t<\frac1k$ the analytic set $C^+\cap C^-$ of real 
dimension $k-2$ has no singularities. This follows from the fact that 
$C^+$ and $C^-$ are parts of analytic sets which are smooth everywhere
except for $1^k$ and $(-1)^k$. Let us emphasize at this point that the intersection of the two cones $C^\pm$ does not need to be contained in a hyperplane, as it can be seen by looking at the large $t$ case, when $S_t$ is the $\ell^\iy$ ball. 

Let us make a list of the smoothness at the possible faces of $S_t$:
\begin{enumerate}
\item When $t\ge\frac{k-1}{k}$, the set $\{x\in\mathbb R^k\colon \normt{x}=1\}$ is simply the $\ell^\infty$ unit ball;
\item When $t\in(1-j/k,1-(j-1)/k)$, 
a point belonging to the relative interior of an exposed face of dimension 
$k-l$ has $k-l$ directions of smoothness for each $k-1\ge l\ge j$. 
There are zero dimensional exposed faces with no direction
of smoothness along $S_t$.
\item When $t<\frac1k$, there are only exposed faces of dimension 
$0$ and $1$. Two of the faces of dimension zero have exactly $k-1$
violations of smoothness, and infinitely many ones (situated on $C^+
\cap C^-)$ have exactly one. The points in the relative interior of
the one-dimensional faces are smooth.
\end{enumerate}
We would like to emphasize that only exposed faces of dimension 1 and
$k-1$ contain points in which $S_t$ is smooth. In addition, in terms of
probability measures $\mu_x$, we note that all points of non-smoothness
on $S_t\setminus(C^+\cap C^-)$ come from surviving atoms of
$\mu_x^{\boxplus 1/t}$. In particular, if $t<1-\frac1k$ and $x_1<x_2<
\cdots<x_k$, then $S_t$ must be smooth at $x$. Recall that $A=\{x\in
\mathbb R^k\colon\normt{x}\leq1\}$, $A^*$ denotes its polar dual, and $K_{k,t} = A^* \cap \Delta_k$.

\begin{corollary}\label{cor:63}
The faces  of the set $A^*$ are as follows:
\begin{enumerate} 
\item
For any $t\in(0,1],k\in\mathbb N$, the set $A^*$ contains in its boundary two exposed faces of dimension $k-1$, namely $\varphi(\{1^k
\})$ and $\varphi(\{(-1)^k\})$. 
\item
When $t\in(1-j/k,1-(j-1)/k)$, the set $A^*$ has in addition
exposed faces of dimensions $l-1$ for any $l\in\{j,\dots,k-1\}.$
\item
In 
particular, when $t\ge\frac{k-1}{k}$, $A^*$ coincides with the
unit ball in the norm one.
\item When $t<\frac1k$, exposed faces of 
$A^*$ are {\rm (I)} $\varphi(\{1^k\})$ and $\varphi(\{(-1)^k\})$ which are two hyperplanes, {\rm (II)} $\varphi(\mathfrak s)$, where 
$\mathfrak s$ is a segment uniting a vertex with a point from $C^+\cap 
C^-$; each $\varphi(\mathfrak s)$ is a point, so their union is 
$k-2$-dimensional and smooth in those $k-2$ directions, and 
{\rm (III)} $\varphi(\{c\})$, for all
$c\in C^+\cap C^-$; since in points of $C^+\cap C^-$ the 
$\normt{\cdot}$-unit ball is smooth in all but one direction, 
each $\varphi(\{c\})$ is a segment, and their union is a smooth $k-1
$-dimensional manifold. Moreover, for any $t<\frac{k-1}{k}$,
the set $A^*$ has infinitely many exposed faces of dimension
zero (i.e. points).
\end{enumerate}
\end{corollary}

Clearly, the second part of the above corollary is not expressed in its 
full strength. However, the number of particular cases that would
need to be treated make a more detailed discussion too involved
to be worth pursuing here. Its proof is a straightforward consequence of the above theorem
and the remarks preceding it.

Finally, it is worth noting that $\varphi(\{1^k\})=\{x\in\mathbb R\colon
\sum x_j=1,\langle x,a\rangle\le1\textrm{ for all }a\in A\}$,
so that $K_{k,t} = \Delta_k\cap A^*\subset\varphi(\{1^k\})$. A point in 
$\Delta_k$ with strictly decreasing coordinates which is on the 
boundary of $A^*$ relative to $\Delta_k$ 
will then be a smooth point for this boundary.
Indeed, assume $x$ is such a point. We know from the previous
theorem and corollary that $x$ cannot be a smooth point of $\partial 
A^*$. Since it must belong to the relative interior of an exposed
face and it does belong to the relative {\em boundary} of 
$\varphi(\{1^k\}),$ it is clear that there is at least one other face 
of $A^*$ to which $x$ belongs, so that there is at least one more 
point $\alpha\in A\setminus\{1^k\}$ (more precise $\alpha\in
\varphi^{-1}(B)$ for some face $B\neq\varphi(\{1^k\})$)
so that $\sum x_j\alpha_j=1$ and 
$\sum x_ja_j\leq 1$ for all other $a\in A$. We claim that this point
$\alpha$ must (a) be unique up to convex combinations with $1^k$, and 
(b) have decreasing coordinates. Indeed, assuming we have an
$\alpha$ satisfying these conditions which does not have 
decreasing coordinates, then we can re-arange it so that its
coordinates do decrease. Its $(t)$ norm will not change, but
its scalar product with $x$ will strictly increase from 1, 
contradicting the definition of $A$ and $A^*$. Also,
$\sum x_j(s\alpha_j+1-s)\equiv 1$ for all $s\in[0,1]$, so the 
lack of uniqueness is proved. Now, finally, we need to argue that
this is the only possible lack of uniqueness. In order to
show that, it is enough to argue that $S_t$ is smooth around $\alpha$,
or, equivalently, that $\normt{\alpha}$ is not reached at an atom.
If this were to happen, then we would have $1=\alpha_1=\cdots=\alpha_j>
\alpha_{j+1}\ge\cdots\ge\alpha_k$ (we know that at least one of the
inequalities is strict because $\alpha\neq1^k$.) Then
$1=\langle\alpha,x\rangle=x_1+\cdots+x_j+\alpha_{j+1}x_{j+1}+\cdots+
\alpha_kx_k<\sum x_j=1,$ an obvious contradiction. Thus, by the
Theorem \ref{convex-norm}, $S_t$ is smooth at $\alpha$, so $\varphi
([\alpha,1^k])=\{x\}$ is an exposed face. 

The above discussion has as an immediate consequence
the following remark:

\begin{remark}\label{lem:exposed-Kkt}
Let $a \in C$ be a non-degenerate direction of the canonical Weyl chamber $\Delta_k^\downarrow$ and $t<1-\frac1k$. 
Then the set $H(a, t) \cap K_{k,t}$ is a singleton.
\end{remark}
We note that this result cannot be improved, as $K_{k,t}
= \Delta_k$
when $t>\frac{k-1}{k}$.

\section*{Acknowledgments}

This project was initiated at the Fields Thematic Program on Mathematics in Quantum Information in 2009. The three authors would like to thank the organizers of the program for providing them an inspiring working environment. This project also benefited from a visit at the Perimeter Institute during the conference ``Random Matrix Techniques in Quantum Information Theory'' in July 2010, and also from visits to their respective authors' universities. 

S. B.'s research was supported by a Discovery grant from the Natural Science and Engineering Research Council of Canada and a University of Saskatchewan start-up grant. B.C. was partly funded by ANR GranMa and ANR Galoisint. The research of I.N. was supported by a PEPS grant from the Institute of Physics of the CNRS and the ANR project ANR 2011 BS01 008 01. The research of B.C. and I.N. was supported in part by NSERC discovery grants, the University of Ottawa and an ERA. Both B.C. and I.N. acknowledge the hospitality of the Mittag-Leffler Institute, where some of this work was done during the ``Quantum Information Theory'' program.


\begin{thebibliography}{99}

\bibitem{akhieser} {N. I. Akhieser}. {\em  The classical moment
problem and some related questions in analysis},
{Hafner Publishing Co.}, {New York}, {(1965)}.

\bibitem{agz} {Anderson, G.,  Guionnet, A. and Zeitouni, O.}.
{\em An Introduction to Random Matrices} ,
{Cambridge University press} {(2010)}

\bibitem{aihp} Belinschi, S. T. {\em A note on regularity for free 
convolutions}. Ann. Inst. H. Poincar\'e Probab. Stat.
{\bf 42}(3) 635--648. 

\bibitem{ptrf} Belinschi, S. T. {\em The Lebesgue decomposition
of the free additive convolution of two probability distributions.}
Probab. Theory Relat. Fields (2008) 142: 125--150.

\bibitem{BB-mathz}
Belinschi, S. T. and Bercovici, H. {\em Atoms and regularity for measures in a partially
defined free convolution semigroup}. Math. Z. 248 (2004), 665--674.

\bibitem{BB-imrn}
Belinschi, S. T. and Bercovici, H. {\em Partially defined semigroups
relative to multiplicative free convolution.}
Int. Math. Res. Not. (2): 65--101, 2005

\bibitem{BB}
Belinschi, S. T. and Bercovici, H. {\em A new approach to 
subordination results in free probability.} Journal d'Analyse 
Math\'ematique, Vol 101 (2007), 357--365.

\bibitem{benaych-rao}
Benaych-Georges, F. and Rao, R.
{\it The eigenvalues and eigenvectors of finite, low rank perturbations of large random matrices.}
arXiv:0910.2120v2, to appear in Adv. in Math.

\bibitem{benaych}
Benaych-Georges, F.
{\it Eigenvectors of Wigner matrices: universality of global fluctuations.}
arXiv:1104.1219.


\bibitem{BVIUMJ} H. Bercovici and D. Voiculescu. {\em Free convolution
of measures with unbounded support}. Indiana Univ. Math. J.
{\bf 42} (3) 733--773 (1993).


\bibitem{BV}
Bercovici, H. and Voiculescu, D. {\em Superconvergence to the central limit and failure of Cramer's Theorem for free random variables.}
Probab. Theory Related Fields {\bf 103} (1995) 215--222.

\bibitem{bhatia}
Bhatia, R. {\it Matrix Analysis}. Graduate Texts in Mathematics, 169. Springer-Verlag, New York, 1997.

\bibitem{Biane} 
Biane, Philippe. {\em Processes with free increments.
} Math. Z. {\bf 227}(1), 143--174 (1998).

\bibitem{chirka} Chirka, E. M. {\em Complex Analytic Sets.} Kluwer
Academic Publishers, 1989.

\bibitem{CollingwoodL} Collingwood, E. F., Lohwater, A. J.: {
The theory of cluster sets}. {Cambridge Tracts in Mathematics and Mathematical 
Physics}, No. 56 Cambridge University Press, Cambridge {(1966)} 

\bibitem{collins-imrn}
Collins, B.
{\it Moments and Cumulants of Polynomial random variables on unitary groups, 
the Itzykson-Zuber integral and free probability }
Int. Math. Res. Not., (17):953-982, 2003. 

\bibitem{collins-male}
Collins, B. and Male, C.
{\it The strong asymptotic freeness of Haar and deterministic matrices.} 
arXiv:1105.4345

\bibitem{collins-nechita-1}
Collins, B. and Nechita, I. 
{\it Random quantum channels I: Graphical calculus and the Bell state phenomenon.} 
Comm. Math. Phys. 297 (2010), no. 2, 345--370.

\bibitem{collins-nechita-2}
Collins, B. and Nechita, I. 
{\it Random quantum channels II: Entanglement of random subspaces, Renyi entropy estimates and additivity problems.} 
Advances in Mathematics 226 (2011), 1181-1201.

\bibitem{collins-sniady}
Collins, B. and \'Sniady, P. {\it Integration with respect to the Haar measure on unitary, orthogonal and symplectic group.} Comm. Math. Phys. 264 (2006), no. 3, 773--795. 

\bibitem{defosseux}
Defosseux, M.
{\it Orbit measures, random matrix theory and interlaced determinantal processes. }
Ann. Inst. Henri Poincar\'e Probab. Stat. 46 (2010), no. 1, 209--249.

\bibitem{Garnett} J. B. Garnett, {\em Bounded analytic functions}, Academic Press, New
York, 1981.

\bibitem{HT}
Haagerup, U. and Thorbj\o{}rnsen, S. 
{\it A new application of random matrices: ${\rm Ext}(C\sp *\sb {\rm red}(F\sb 2))$ is not a group.}  
Ann. of Math. (2)  162  (2005),  no. 2, 711--775.

\bibitem{hayden-winter}
Hayden, P. and Winter, A. 
{\it Counterexamples to the maximal p-norm multiplicativity conjecture for all $p>1$}. 
Comm. Math. Phys. 284 (2008), no. 1, 263--280.

\bibitem{male}
Male, C. 
{\it Norm of polynomials in large random and deterministic matrices.} 
arXiv:1004.4155v2.

\bibitem{metcalfe}
Metcalfe, T.
{\it Ph.D. dissertation}

\bibitem{nader}
Nadler, B.
{\it Finite sample approximation results for principal component analysis: a matrix perturbation approach.}
 Ann. Statist. 36 (2008), no. 6, 2791--2817.

\bibitem{NS-mult} {Nica, A. and Speicher, R.}, {\em On the multiplication of free
$n$-tuples of noncommutative random variables}, {Amer. J. Math.} 
{\bf 118} {(1996)}, {799--837}.

\bibitem{NS} 
Nica, A. and Speicher, R. {\em Lectures on the combinatorics of free probability}, Cambridge Univ. Press (2006).

\bibitem{rockafellar}
Rockafellar, R. T.
{\it Convex analysis.}
Princeton Mathematical Series, No. 28 Princeton University Press, Princeton, N.J. 1970 xviii+451 pp. 

 \bibitem{Shapiro}
 Shapiro, J. H.  {\em Composition operators and classical function
theory}, Springer, New York, 1993.

\bibitem{V-boxplus} Voiculescu, D. {\it Addition of certain 
noncommuting random variables.} J. Funct. Anal. {\bf 66}, 323--346
(1986)

\bibitem{V-boxtimes} Voiculescu, D. {\it Multiplication of certain 
noncommuting random variables.} J. Oper. Theory (1987)

\bibitem{voiculescu-dykema-nica}
Voiculescu, D.V.,  Dykema. K.J. and Nica, A.
{\it Free random variables}, AMS (1992).

\bibitem{V-fe} Voiculescu, D. {\it The analogues of entropy and of
Fisher's information measure in free probability theory. 1.}
Comm. Math. Phys. {\bf 155}, 71--92
(1993)

\bibitem{webster} Webster, Roger J. {\em Convexity}. Oxford University Press, New York, 1994.

\end{thebibliography}
\end{document}